\documentclass[reqno]{amsart}

\usepackage{amsmath}
\usepackage{amssymb}
\allowdisplaybreaks
\RequirePackage{mathtools, amsthm, graphicx, mathrsfs, dsfont}

\newcommand{\A}{\mathcal{A}}
\newcommand{\C}{\mathbb{C}}

\newcommand{\N}{\mathbb{N}}

\newcommand{\Q}{\mathbb{Q}}
\newcommand{\R}{\mathbb{R}}
\renewcommand{\S}{\mathcal{S}}

\newcommand{\X}{\mathfrak{X}}
\newcommand{\Z}{\mathbb{Z}}
\newcommand{\1}{\mathbf{1}}
\let\epsilon\varepsilon

\providecommand{\norm}[1]{\lvert\lvert #1 \rvert\rvert}
\newtheorem{theorem}{Theorem}[section]
\newtheorem{lemma}[theorem]{Lemma}

\newtheorem{prop}[theorem]{Proposition}
\newtheorem{cor}[theorem]{Corollary}

\theoremstyle{definition}
\newtheorem{definition}[theorem]{Definition}
\newtheorem{example}[theorem]{Example}

\theoremstyle{remark}
\newtheorem{remark}[theorem]{Remark}

\numberwithin{equation}{section}

\newcommand{\abs}[1]{\lvert#1\rvert}
\newcommand{\floor}[1]{\lfloor#1\rfloor}
\newcommand{\ceil}[1]{\lceil#1\rceil}

\begin{document}

\title[]{Modulus of continuity for spectral measures of suspension flows over Salem type substitutions}

\author{Juan Marshall-Maldonado}
%\curraddr{Department of Mathematics and Statistics,
%Case Western Reserve University, Cleveland, Ohio 43403}
\address{Juan Marshall-Maldonado, Department of Mathematics,
	Bar-Ilan University, Ramat-Gan, Israel}
\email{jgmarshall21@gmail.com}
%\thanks{The author was supported in part by Center for Mathematical Modeling (CMM)}

%    General info
%\subjclass[2000]{Primary 54C40, 14E20; Secondary 46E25, 20C20}

%\date{Some date,201? and in revised form some date, 201?.}

%\dedicatory{This paper is dedicated.}

%\keywords{Spectral measure, substitution tilings}

\begin{abstract}
We study the spectrum of the self-similar suspension flows of subshifts arising from primitive substitutions. We focus on the case where the substitution matrix has a Salem number $\alpha$ as dominant eigenvalue. We obtain a H\"{o}lder exponent for the spectral measures for points away from zero and belonging to the field $\Q(\alpha)$. This exponent depends only on three parameters of each of these points: its absolute value, the absolute value of its real conjugate and its denominator.
\end{abstract}
\maketitle
%\section*{This is an unnumbered first-level section head}
%This is an example of an unnumbered first-level heading.

%% The correct journal style for \specialsection is all uppercase; a known bug
%% in amsart.cls prevents this, so input must be uppercase until it is fixed.
%\specialsection*{This is a Special Section Head}
%\specialsection*{THIS IS A SPECIAL SECTION HEAD}
%This is an example of a special section head%
%%%%%%%%%%%%%%%%%%%%%%%%%%%%%%%%%%%%%%%%%%%%%%%%%%%%%%%%%%%%%%%%%%%%%%%%
%\footnote{Here is an example of a footnote. Notice that this footnote
%text is running on so that it can stand as an example of how a footnote
%with separate paragraphs should be written.
%\par
%And here is the beginning of the second paragraph.}%
%%%%%%%%%%%%%%%%%%%%%%%%%%%%%%%%%%%%%%%%%%%%%%%%%%%%%%%%%%%%%%%%%%%%%%%%
\section{Introduction}
Substitutions appear as natural objects in many different research areas such as symbolic dynamics, number theory, combinatorics of words, Diophantine approximation, and so on. For instance, the study of the spectrum of dynamical systems arising from substitutions has left longstanding open problems. One of the most important is the \textit{Pisot substitution conjecture}, which asserts that if a substitution is irreducible and of \textit{Pisot type}, then the corresponding subshift has pure discrete spectrum (see \cite{akiyama2015pisot}).

A continuous counterpart is found studying tilings of the Euclidean space, and similar questions for the spectrum emerge again in this case. A suspension flow of a substitution subshift can be seen as a special kind of tiling of the real line. Previous work on this type of systems are the papers of Clark and Sadun \cite{ClarkSadun}, \cite{clark_sadun_2006} and Barge and Diamond \cite{barge2008cohomology}. By relating the eigenvalues of the system with the eigenvalues of the matrix representing the cohomological action of the substitution map, it is possible to conclude that generic tile suspensions over non Pisot irreducible substitutions are weakly mixing. In this case, the spectral measures do not have any atoms (except for the trivial one at the origin), and moduli of continuity of these measures are linked with rates of weak mixing (see \cite{Knill}). Decay rates of spectral measures also give information on its absolutely continuous and singular components.

In the weak mixing case, the work of Bufetov and Solomyak \cite{bufetov2014modulus} analyses the case in which the Perron-Frobenius eigenvalue of the substitution matrix has at least one conjugate outside the closed unit disk. They prove a H\"{o}lder decay of the spectral measures for a typical suspension flow (if the characteristic polynomial of the substitution matrix is irreducible), and a log-H\"{o}lder one for a self-similar suspension flow with this hypothesis. In both cases, for spectral parameters away from the origin. They have broadened many of the tools used in that article from this simple setting to more complex systems such as the ones associated to Bratelli-Vershik diagrams or translations flows on flat surfaces (see \cite{bufetov2018holder},\cite{JEP_2021__8__279_0}). A different approach in this last setting to this problem is found in \cite{forni2019twisted}, and the study for higher rank actions is in \cite{trevino2020quantitative}.

The main objective of this paper is to study the spectrum of the self-similar suspension flow when the substitution is of Salem type, i.e., when the dominant eigenvalue of the substitution matrix is a Salem number. This is a question raised in \cite{bufetov2014modulus}. This could be thought as a limit case, since we do not have homoclinic points (which give rise to eigenvalues in the Pisot case) nor an unstable subspace ensuring the absence of atoms (e.g., if there is a different conjugate of the Perron-Frobenius eigenvalue of the substitution matrix outside the unit disk). Instead, it acts as an isometry on the invariant subspace complementary to the subspace generated by the eigenvectors associated to the dominating eigenvalue and its inverse.

Salem substitutions arise naturally in the study of Veech groups: on each non arithmetic primitive Veech surface of genus two, there exists a pseudo-Anosov diffeomorphism whose dilatation is a Salem number of degree 4. The associated interval exchange transformation (by zippered rectangles) is self-similar and is (more precisely, the loop in its Rauzy diagram) defined by a substitution of Salem type (see \cite{bressaud2014deviation}). Explicit examples appear in \cite{bressaud2014deviation} and \cite{avila2016weak}.

The difference between the analogous result in \cite{bufetov2014modulus} and ours is reflected in the complicated dependence we have found on the parameters controlling the decay of the spectral measure. In fact, we are only able to find this decay when the spectral parameter belongs to the number field generated by the principal eigenvalue, as we show in the next result. For integers $l_0,\dots, l_n$, we denote by $\gcd(l_0,\dots,l_n)$ the greatest common divisor of $l_0,\dots, l_n$. Also, for a number $\eta = (l_0 + \dots + l_{d-1}\alpha^{d-1})/L$ (with $L\in\N$ and $l_0,\dots, l_{d-1}\in\Z$) belonging to a number field $\Q(\alpha)$, we say it is in \textit{reduced form} if $\gcd(l_0,\dots,l_{d-1},L)=1$. 
\begin{theorem}\label{theorem1}
	Let $\zeta$ be a Salem type, aperiodic and primitive substitution on $\A$, $\alpha$ its Perron-Frobenius eigenvalue and $\vec{p}$ the positive (left-)eigenvector of the substitution matrix. Let $\mathfrak{X}^{\vec{p}}_{\zeta}$ be the corresponding self-similar suspension flow and for any $a \in \A$, denote by $\nu_a$ the correlation measure associated to $a$. Consider $\sigma_0 :\Q(\alpha)\hookrightarrow \C$ the embedding corresponding to $\alpha\mapsto\alpha^{-1}$. Fix $A,B,C > 1$ and suppose $\abs{\omega} \in \Q(\alpha)\cap[ B^{-1},B]$ satisfies $\abs{\sigma_0(\omega)}\leq C$ and $L\leq A$, where $L \in \N$ is defined by the expression in reduced form $\omega = \dfrac{1}{L}(l_0 + \dots + l_{d-1}\alpha^{d-1})$. Then there exist $\gamma = \gamma(A,B,C),\, c=c(\zeta),\, r_0=r_0(\omega)>0$ such that
	\[
	\nu_a([\omega-r,\omega+r]) \leq cr^{\gamma},
	\]
	for all $0< r < r_0$ and $a\in\A$.
\end{theorem}
Under an arithmetic condition there is a uniform dependence of $\gamma$ on the variables $B,C$, as we state in our second result.
\begin{theorem}\label{theorem2}
	Let $\zeta$ be a Salem type, aperiodic and primitive substitution and $\nu_a$ the correlation measure associated to the letter $a\in\A$ on the self-similar suspension flow.  There exists $\kappa\in\Q(\alpha)$ an explicit positive constant such that the next statement holds: for a fixed $A > 1$, suppose
	\begin{itemize}
		\item There exists $n\in\{0,\dots,d-1\}$ such that $\text{Tr}(L\kappa\omega\alpha^n)\not\equiv 0 \: (\text{mod }L)$, and
		\item $L\leq A$,
	\end{itemize}
	where $L \in \N$ is defined by the expression in reduced form $0\neq\omega \kappa = \dfrac{1}{L}(l_0 + \dots + l_{d-1}\alpha^{d-1})$. Then there exist $\gamma = \gamma(A),\, c=c(\zeta),\, r_0=r_0(\omega)>0$ such that
	\[
	\nu_a([\omega-r,\omega+r]) \leq cr^{\gamma},
	\]
	for all $0< r < r_0$, $a\in\A$.
\end{theorem}
The first condition above is generic, in the sense that the complement of the set of $\omega$'s satisfying it, is contained in a finite union of lattices of $\Q(\alpha)$, according to
\begin{prop}\label{lattice}
	Suppose $\eta = \dfrac{1}{L}(l_0 + \dots + l_{d-1}\alpha^{d-1}) \in \Q(\alpha)$ is in reduced form and $\text{Tr}(L\eta\alpha^n)\equiv 0 \: (\text{mod }L)$ for all $n=0,\dots,d-1$. Then $L$ divides $E(\alpha)$, where $E(\alpha)$ is the least common multiple of the denominators of the dual basis of $\{1,\alpha,\dots,\alpha^{d-1}\}$ (expressed in reduced form).
\end{prop}
An important difference with the results in \cite{bufetov2014modulus} is the complex dependence of $r_0>0$ on $\omega$, in both Theorem \ref{theorem1} and \ref{theorem2}. A lower bound is derived explicitly when $\deg(\alpha)=4$, but it does not have a simple expression in terms of the spectral parameter. We will summarize this in Proposition \ref{propr_0}.

The proof of Theorems \ref{theorem1} and \ref{theorem2} is based in the distribution modulo 1 of the sequence $(\omega\alpha^n)_{n\geq0}$. The relation between a modulus of continuity for the spectral measures and the distribution modulo one of such a sequence may be seen from Lemma \ref{lemmaprod} and Proposition \ref{Spectralbound}. S. Akiyama and Y. Tanigawa showed in \cite{akiyama2004salem} that the sequence $(\alpha^n)_{n\geq1}$ is not far from being uniformly distributed modulo 1. We recall this result in Theorem \ref{akiyama} . We are able to prove a similar result for the sequences $(\omega\alpha^n)_{n\geq0}$, for $\omega \in \Q(\alpha)$.
\begin{theorem}
	Denote $J(\delta) = [\delta,1-\delta]$, for $\delta < 1/2$. Let $\alpha$ be a Salem number of degree $d$ and $\eta= \dfrac{1}{L}(l_0 + \dots + l_{d-1}\alpha^{d-1})\in\Q(\alpha)$, with $L\geq1$ and $l_0,\dots,l_{d-1}\in\Z$. Consider $\sigma_0 :\Q(\alpha)\hookrightarrow \C$ the embedding corresponding to $\alpha\mapsto\alpha^{-1}$. Assume $\gcd(l_0,\dots,l_{d-1},L) = 1$. Then there exists an explicit $\delta = \delta(L,\abs{\eta},\abs{\sigma_0(\eta)}) > 0$ such that
	\[
	\lim_{N\to\infty} \dfrac{\#\left\{n\leq N \,\middle|\, \left\{\eta\alpha^n\right\}\in J(\delta)\right\}}{N} \geq 1/2.
	\]
\end{theorem}
This result is easily deduced from Corollary \ref{cor} and from Lemmas \ref{lemma1}, \ref{lemma2}, \ref{lemma3} and \ref{lemma4}.

The rest of the paper is organized as follows. In Section 2 we provide a background material around substitutions, spectral theory of dynamical systems, algebraic number theory and harmonic analysis. In particular, we recall the definition of special trigonometric polynomials used in the proof of Theorems \ref{theorem1} and \ref{theorem2}. In Section 3 we sketch the proof of the two main theorems since it is rather technical in full generality. Section 4 is devoted to the proof of Theorems \ref{theorem1} and \ref{theorem2}. Section 5 is devoted to prove Proposition \ref{lattice}. Finally, in Section 6 we study the nature of $r_0$ appearing in both main results, ending with the proof of Proposition \ref{propr_0}.

\subsection*{Acknowledgements}
The author was financially supported by the ERC project No 647133 (ICHAOS). The author would like to thank Shigeki Akiyama for providing useful observations leading to the final form of Proposition \ref{lattice}. The author would also like to thank Lorenzo Sadun for his observations about the spectrum of generic tile length substitutions, and Carlos Matheus and Boris Solomyak for carefully reading the first version of the manuscript. The author would like also to thank S\'ebastien Gou\"ezel and the anonymus referee for the numerous commentaries which lead to improve the presentation. Finally, the author would like to thank Pascal Hubert and Alexander Bufetov for the numerous discussions around this problem and constant encouragement. 

\section{Background}
\subsection{Dynamical systems arising from substitutions}
The basic notions of substitutions may be found in \cite{queffelec2010substitution,Fogg} with more detail. Let us start by fixing a positive even integer $d\geq 4$ and a finite alphabet $\A = \left\{1,\dots,d\right\}$. A \textit{substitution} on the alphabet $\A$ is a map $\zeta : \A \longrightarrow \A^+$, where $\A^+$ denote the set of finite (nonempty) words on $\A$. By concatenation, it is natural to extend a substitution to $\A^+$, to $\A^\N$ (one-sided sequences) or $\A^\Z$ (two-sided sequences). In particular, the iterates $\zeta^n(a) = \zeta(\zeta^{n-1}(a))$ for $a \in \A$, are well defined.
\begin{example}[see \cite{holton1998geometric}]\label{example}
	Let $\A = \left\{1,2,3,4\right\}$ and define $\zeta$ by
	\begin{align*}
	&\zeta(1) = 12, \quad \zeta(3) = 2,\\
	&\zeta(2) = 14, \quad \zeta(4) = 3.
	\end{align*}
\end{example}
For a word $w\in\A^+$ denote its length by $\abs{w}$ and by $\abs{w}_a$ the number of symbols $a$ found in $w$. The \textit{substitution matrix} associated to a substitution $\zeta$ is the $d\times d$ matrix with integer entries defined by $M_{\zeta}(a,b) = \abs{\zeta(b)}_a$. A substitution is called \textit{primitive} if its substitution matrix is primitive.
\begin{example}
	Let $\zeta$ be the substitution defined in Example 2.1. Then its substitution matrix is 
	\[
	M_{\zeta}=\begin{pmatrix}1&1&0&0\\1&0&1&0\\0&0&0&1\\0&1&0&0\end{pmatrix},
	\]
	and this substitution is primitive.
\end{example}
The \textit{substitution subshift} associated to $\zeta$ is the set $\mathfrak{X}_\zeta$ of sequences $(x_n)_{n\in\Z} \in \A^\Z$ such that for every $i\in\Z$ and $k\in\N$ exist $a \in \A$ and $n\in\N$ such that $x_i\dots x_{i+k}$ is a subword of some $\zeta^n(a)$. A classical result is that the $\Z$-action by the \textit{left-shift} $T((x_n)_{n\in\Z}) = (x_{n+1})_{n\in\Z}$ on this subshift is minimal and uniquely ergodic when $\zeta$ is primitive. From now on we only consider primitive and \textit{aperiodic} substitutions, which means in the primitive case that the subshift is not finite.

Now we turn to the continuous counterpart of the substitution subshift. The study of the discrete part of the spectrum of the suspensions over substitutions (not only the self-similar case) is done in the article \cite{ClarkSadun}. For a primitive substitution $\zeta$, denote by $\vec{p}=(p_a)_{a\in\A}$ a positive Perron-Frobenius left-eigenvector of $M_\zeta$. For convenience, we will normalize the vector $\vec{p}$ in such a way that its components belong to $\Z[\alpha]$ (it is enough to multiply for some positive integer). This allows us to assume the constant $\kappa$ appearing in Lemma \ref{lemmaprod} belongs to $\Z[\alpha]$ (see the remark below Lemma \ref{lemmaprod}). Set $F: \X_\zeta \times \R \longrightarrow \X_\zeta \times \R$ defined by $F(x,t) = (T(x),t-p_{x_0})$. 
\begin{definition}\label{suspension}
	The \textit{self-similar suspension flow} is the pair $(\mathfrak{X}^{\vec{p}}_{\zeta},(h_t)_{t\in\R})$ given by
	\begin{align*}
	\mathfrak{X}^{\vec{p}}_{\zeta} &= (\X_\zeta \times \R) /\sim,\\
	h_t(x,t') &= (x,t'+t) \quad (\text{mod }\sim),
	\end{align*}
	where $\sim$ is the equivalence relation defined by $(x,t) \sim (x',t')$ if and only if $F^n(x,t) = (x',t')$, for some $n \in \Z$. 
\end{definition}
We will identify $\mathfrak{X}^{\vec{p}}_{\zeta}$ with a \textit{fundamental domain} (see for example \cite{Viana}, Chapter 3). We will take as fundamental domain 
\[
\mathcal{D} = \left\{ (x,t)\in  \mathfrak{X}^{\vec{p}}_{\zeta} \times \R \,\mid \, 0\leq t < f(x) \right\}.
\] 
We will make no further reference to the fundamental domain, and we will simply denote it by $\mathfrak{X}^{\vec{p}}_{\zeta}$. We may decompose
\[
\mathfrak{X}^{\vec{p}}_{\zeta} = \bigcup_{a\in\A} \mathfrak{X}^{\vec{p}}_a, \quad \mathfrak{X}^{\vec{p}}_a = \left\{ (x,t)\in\mathfrak{X}^{\vec{p}}_{\zeta}\: \middle| \: x_0 = a \right\}.
\]

Once again, this flow is uniquely ergodic and the only (Borel) measure invariant for the flow $(h_t)_{t\in\R}$ will be denoted by $\mu$. Our results will concern the spectral measures (we recall its definition in the next subsection) associated to the indicator functions of this partition (in measure), i.e., $f = \mathds{1}_{{\mathfrak{X}^{\vec{p}}_a}}$, for each $a\in\A$.
\subsection{Spectral theory}
We define the main objects of study of this work, namely, the spectral measures associated to the self-similar suspension flow. We restrict ourselves to stating a dynamical version of the spectral theorem taken from \cite{KTSpectral}. A more extensive introduction may also be found in \cite{queffelec2010substitution,Fogg}.
\begin{theorem}\label{spectralthm}
	Let $f,g \in L^2(\mathfrak{X}^{\vec{p}}_{\zeta},\mu)$. There exists a complex measure $\nu_{f,g}$ with support in $\R$, called \textbf{spectral measure}, such that for all $t\in \R$
	\[
	\widehat{\nu_{f,g}}(-t) := \int_{\mathbb{\R}}e^{2\pi it\omega}d\nu_{f,g}(\omega) = \langle U^tf, g \rangle_{L^2(\mathfrak{X}^{\vec{p}}_{\zeta},\mu)},
	\]
	where $U^tf(y) = f(h_t(y))$ for $y \in \mathfrak{X}^{\vec{p}}_{\zeta}$.
\end{theorem}
We will denote $\nu_{f,f}$ by $\nu_f$, and for $f = \mathds{1}_{{\mathfrak{X}^{\vec{p}}_a}}$ we will write $\nu_{f}=\nu_a$. The measure $\nu_a$ will be called the \textit{correlation measure} associated to $a$. To study the asymptotics of spectral measures, we look into a special kind of Birkhoff integral.
\begin{definition}
	Let $f \in L^2(\mathfrak{X}^{\vec{p}}_{\zeta},\mu)$, $(x,s)\in \mathfrak{X}^{\vec{p}}_{\zeta}$, $\omega\in\R$, $R>0$. The \textit{twisted Birkhoff integral} associated to $f$ is defined by 
	\[
	S^f_R((x,s),\omega) = \int_0^R e^{-2\pi i \omega t} f(h_t(x,s)) dt.
	\]
\end{definition}
The relation between these two concepts is clarified by the next proposition.
\begin{prop}[\cite{bufetov2014modulus}, Lemma 4.3]\label{Spectralbound}
	Denote $G_R(f,\omega) = \dfrac{1}{R}\norm{S^f_R(\cdot,\omega)}_{L^2}^2$. Suppose there exists $0<\gamma<1$ such that $G_R(f,\omega) \leq CR^{2\gamma-1}$, for some constant $C>0$ and $R\geq R_0$. Then there exists $r_0>0$ (depending only on $R_0$) such that for every $0<r\leq r_0$ holds
	\[
	\nu_f([\omega-r,\omega+r]) \leq \pi^2Cr^{2(1-\gamma)}.
	\]
\end{prop}
Finally, the problem of finding bounds for the twisted Birkhoff sums may be addressed solving a problem on Diophantine approximation, according to the next lemma.  We denote the distance of $x\in\R$ to the nearest integer by $\norm{x}_{_{\R/\Z}} = \min(\left\{x\right\},1-\left\{x\right\})$ , where as usual $\left\{x\right\} = x-\floor{x}$ denotes the fractional part of $x$.
\begin{lemma}[\cite{bufetov2014modulus}, Proposition 4.4]\label{lemmaprod}
	Let $\zeta$ be a primitive substitution (with $\alpha$ being its Perron-Frobenius eigenvalue and $\vec{p}$ the eigenvector normalized as before) and $\mathfrak{X}^{\vec{p}}_{\zeta}$ the corresponding self-similar suspension flow. Let $a\in\A$ and $f$ be the indicator function of $\mathfrak{X}^{\vec{p}}_a$. Then there exist $\lambda \in (0, 1)$, $C_1>0$ and $\kappa \in \Z[\alpha]$ an explicit positive constant, all depending only on the substitution $\zeta$, such that
	\[
	\abs{S^f_R((x,s),\omega)} \leq C_1R\prod_{n=0}^{\floor{\log_{\alpha}(R)}} (1-\lambda\norm{\omega \kappa\:\alpha^n}^2_{_{\R/\Z}}),
	\]
	for all $R>0$, $(x,s) \in \mathfrak{X}^{\vec{p}}_{\zeta}$ and $\omega\in\R$.
\end{lemma}
\begin{remark}\label{remarkKappa}
	As we remarked before, if $\vec{p}\in(\Z[\alpha])^d$, then we may assume $\kappa \in \Z[\alpha]$, since by definition $\kappa = \langle \text{Ab}(w),\vec{p}\rangle$, where $w$ is an appropriate \textit{return word} and $\text{Ab}(w)$ its \textit{population vector} indexed by $\A$ with components $(\text{Ab}(w))_a = \abs{w}_a$. See \cite{bufetov2014modulus} for more details.
\end{remark}

\subsection{Salem numbers}
Recall the definition of a \textit{Salem number}: a real algebraic integer  greater than 1 having all its Galois conjugates inside the closed unit disk, with at least one conjugate on the unit circle. This definition actually forces that the inverse is a conjugate, and the rest of them are on the unit circle. For a survey on Salem numbers see \cite{smyth2015seventy}. We will say a primitive substitution is of \textit{Salem type} if the dominant eigenvalue is a Salem number. An example of substitution of Salem type is the one of Example \ref{example}.

We now state some results regarding Salem numbers.
\begin{prop}[see \cite{bugeaud2012distribution}, Theorem 3.9]\label{SalemProp}
	Let $\alpha$ be a Salem number and $\epsilon>0$, then there exists $\eta = \eta(\epsilon) \in \Q(\alpha)$ different from zero such that 
	\[
	\norm{\eta\alpha^n}_{_{\R/\Z}} < \epsilon,
	\]
	for all $n\geq0$.
\end{prop}
In fact, in \cite{zaimi2012comments} there is a characterization of numbers $\eta\in\R$ such that $\limsup_n \norm{\eta\alpha^n}_{_{\R/\Z}} < \epsilon$, with $\epsilon\leq \delta_1(\alpha) = 1/\mathscr{L}(\alpha)$, where $\mathscr{L}(\alpha)$ denotes the \textit{length} of $\alpha$, i.e., the sum of the absolute values of the coefficients of its minimal polynomial. Proposition \ref{SalemProp} shows the difficulty to find a universal exponent for the spectral measure as in \cite{bufetov2014modulus} in the case of Salem type substitutions, at least by the methods we are using.

A classical result states that the sequence $(\alpha^n)_{n\geq1}$ is dense but not uniformly distributed. In spite of this result,  S. Akiyama and Y. Tanigawa showed in \cite{akiyama2004salem} that the sequence $(\alpha^n)_{n\geq1}$ is not far from being uniformly distributed modulo 1.
\begin{theorem}[\cite{akiyama2004salem}, Theorem 1]\label{akiyama}
	Let $\alpha$ be a Salem number of degree $d = 2m+2 \geq 8$ and $J = [\mathfrak{a},\mathfrak{b}] \subseteq [0,1]$, then
	\[
	\left| \lim_{N\to\infty} \dfrac{\#\left\{n\leq N \,\middle|\, \left\{\alpha^n\right\}\in J\right\}}{N} - (\mathfrak{b}-\mathfrak{a}) \right|
	\leq 
	2\zeta\left(m/2\right)(2\pi)^{-m}(\mathfrak{b}-\mathfrak{a}),
	\]
	where $\zeta$ denotes the Riemann zeta function.
\end{theorem}
For degree $d=4$ and $d=6$ there are similar estimates we omit here. An extension of this work is the heart of the proof of Lemmas \ref{lemma1}, \ref{lemma2}, \ref{lemma3} and \ref{lemma4}, which are the essential steps to prove Theorems \ref{theorem1} and \ref{theorem2}.
\begin{prop}[see \cite{bugeaud2012distribution}, proof of Lemma 3.8]\label{ratind}
	Let $\alpha$ be a Salem number of degree $d=2m+2$ and $e^{2\pi i \theta_{1}},\dots,e^{2\pi i \theta_{m}}$ the conjugates on the upper-half of the unit circle. Then $1,\theta_{1},\dots,\theta_m$ are rationally independent.
\end{prop}
Let us recall the notion of \textit{trace} of an algebraic number: for an algebraic number $\eta\in\Q(\alpha)$ define
\[
\text{Tr}(\eta) := \sum_{\sigma:\Q(\alpha)\hookrightarrow\C} \sigma(\eta),
\] 
where the sum runs over all embeddings of $\Q(\alpha)$ in $\C$. In particular, for an algebraic integer $\eta$ we have $\text{Tr}(\eta) \in \Z$.

Finally, we state a classic inequality about evaluation of integer polynomials on algebraic numbers. For $P(X) = a_0 + \dots + a_dX^d \in \Z[X]$ the (naive) \textit{height} of $P$ is defined as Height$(P) = \max(\abs{a_0},\dots,\abs{a_d})$.
\begin{lemma}[\cite{Garsia}, Lemma 1.51]\label{Garsia}
	Let $\xi$ be an algebraic integer and $Q \in \Z[X]$ of degree at most $n\geq 1$ such that $Q(\xi) \neq 0$. Let $\xi_1,\dots,\xi_d$ be the other conjugates of $\xi$ and $m$ the number of $i's$ such that $\abs{\xi_i} = 1$. Then
	\[
	\abs{Q(\xi)} \geq \dfrac{\prod_{\abs{\xi_i} \neq 1} \abs{\abs{\xi_i}-1}}{(n+1)^m \left(\prod_{\abs{\xi_i} > 1} \abs{\xi_i}\right)^{n+1} \textnormal{Height}(Q)^d}.
	\]
\end{lemma}
\subsection{Polynomial approximation of functions}\label{polynomials}
One of the main technical tools we use is a family of trigonometric polynomials (Selberg polynomials) which approximate the indicator function $\mathds{1}_{_J}$ of an interval $J\subset [0,1]$. A detailed reference is found in \cite{montgomery1994ten}, Chapter 1. In order to introduce this family, we recall the definition of several other families, starting with the well-known Fejer kernel.
\begin{definition}
	The \textit{Fejer kernel of degree $N-1$} is defined by
	\[
	\Delta_{N}(z) = \sum_{0\leq\abs{k}\leq N-1} \left(1 - \dfrac{\abs{k}}{N}\right)e^{2\pi i k z}
	\]
\end{definition}
\begin{definition}\label{defValeer}
	The \textit{Vaaler polynomial of degree $N$} is defined by
	\[
	\mathcal{V}_{N}(z) = \dfrac{1}{N+1} \sum_{k=1}^N f\left(\dfrac{k}{N+1}\right)\sin(2\pi k z), 
	\]
	where the function $f$ is given by $f(x) = -(1-x)\cot(\pi x) - 1/\pi$ and satisfies for every $\xi<1/2$ the inequalities
	\begin{align}
	\abs{f(x)} \leq \begin{cases}
	\:\dfrac{\pi\xi}{\sin(\pi\xi)}\dfrac{1}{\pi x} + \dfrac{1}{\pi} & \text{if } 0<x\leq\xi,\\
	\\
	\:\dfrac{1-\xi}{\sin(\pi(1-\xi))} + \dfrac{1}{\pi} & \text{if } \xi<x<1.
	\end{cases}
	\end{align}
	Note also that $\abs{f}$ is decreasing in $(0,1)$.
\end{definition}
\begin{definition}
	The \textit{Beurling polynomial of degree $N$} is defined by
	\[
	\mathcal{B}_N(z) = \mathcal{V}_{N}(z) + \dfrac{1}{2N+1}\Delta_{N}(z).
	\]
\end{definition}
The Beurling polynomials provide an approximation to the sawtooth function $s(x) = \{x\}-1/2$ if $x\notin\Z$, and $s(x)=0$ if $x\in\Z$. The Vaaler lemma (see \cite{montgomery1994ten}, Chapter 1, p. 6) ensures the choice of these polynomials is optimal in certain sense. If we denote by $\mathds{1}_{_J}$ the periodic extension to $\R$ of the indicator function of some interval $J=[\mathfrak{a},\mathfrak{b}]\subseteq[0,1]$, we have the equality $\mathds{1}_{_J}(z) = \mathfrak{b}-\mathfrak{a} + s(z-\mathfrak{b}) + s(\mathfrak{a}-z)$. This fact justifies the next definition.
\begin{definition}
	The \textit{Selberg polynomials of degree $N$ of an interval $J=[\mathfrak{a},\mathfrak{b}]\subseteq[0,1]$ } are defined by
	\begin{align*}
	\mathcal{S}^+_N(z) &= \mathfrak{b}-\mathfrak{a} + \mathcal{B}_N(z-\mathfrak{b}) + \mathcal{B}_N(-z+\mathfrak{a}), \\
	\mathcal{S}^-_N(z) &= \mathfrak{b}-\mathfrak{a} - \mathcal{B}_N(-z+\mathfrak{b}) - \mathcal{B}_N(z-\mathfrak{a}).
	\end{align*}
	
	These polynomials satisfy for all $N\geq 1$,
	\begin{align}
	\S^-_N \leq \mathds{1}_{_J} \leq \S^+_N.\label{propertySelberg}
	\end{align}
\end{definition}

\begin{subsection}{Bessel functions}
	We will denote by $J_0(x)$ the \textit{Bessel function of order zero}, that is, the unique solution to 
	\[
	\begin{cases}
	\:xy'' + y' + xy = 0\\
	\:y(0) = 1\\
	\:y'(0) = 0.
	\end{cases}
	\]
	
	We summarize two classic properties used later in the next
	\begin{prop}[see \cite{akiyama2004salem}, Lemma 2]\label{Bessel}
		For $x\geq0$,
		\[
		J_0(x) = \int_{0}^{1} e^{ix\cos(2\pi t)}dt, \quad \abs{J_0(x)} \leq \min\left(1,\sqrt{\dfrac{2}{\pi x}}\right).
		\]
	\end{prop}
	To finish this background section, we prove an equality involving the Bessel function which will be used repeatedly in the technical calculations from Lemmas \ref{lemma1} to \ref{lemma4}, to calculate the integral of the Beurling polynomial.
	\begin{prop}\label{IntegralBeurling} Let $H_1,\dots,H_m$ be positive real numbers and
		\[
		z(x_1,\dots,x_m) = 2\sum_{j=1}^m H_j\cos(2\pi x_j),
		\]
		with every $x_j\in [0,1]$. Then, for every integer $k\geq 1$, we have
		\[
		\int_{(\R/\Z)^m} e^{2\pi ik z(x_1,\dots,x_m)}dx_1\dots dx_m = \prod_{j=1}^m J_0(4\pi kH_j). 
		\]
	\end{prop}
	\begin{proof}
		\begin{align*}
		\int_{(\R/\Z)^m}  e^{2\pi ik z(x_1,\dots,x_m)}dx_1\dots dx_m &=\prod_{j=1}^{m} \int_{\R/\Z} e^{4\pi ikH_j\cos(2\pi x_j)}dx_j \\
		&= \prod_{j=1}^{m} J_0(4\pi kH_j).
		\end{align*}
	\end{proof}
\end{subsection}

\section{Outline of the proof}
In this section we sketch the basic steps towards the proof of Theorem \ref{theorem1} and Theorem \ref{theorem2} in a simpler case to fix ideas, and leave the formal proof to the next section.

By Lemma \ref{lemmaprod}, we are interested in the distribution of the sequence $(\omega\kappa\alpha^n)_{n\geq0}$ modulo 1. Let us outline the strategy when $\eta = \omega\kappa$ belongs to $\Z[\alpha]$. Let $\alpha$ be a Salem number of degree $d=2m+2$ and $\alpha_1 =\sigma_1(\alpha)= e^{2\pi i \theta_1}, \dots, \alpha_m =\sigma_m(\alpha)= e^{2\pi i \theta_m}$ the Galois conjugates on the upper half of $S^1$. Denote the embeddings of $\Q(\alpha)$ in $\C$ by $\sigma_j$, where $\sigma_0(\alpha) = \alpha^{-1}$ and $\sigma_j(\alpha) = \alpha_j = e^{2\pi i \theta_{j}}$ for $j = 1,\dots,m$.

In this case, for all $n\geq 0$
\[
\text{Tr}(\eta\alpha^n) = \eta\alpha^n + \sigma_0(\eta)\alpha^{-n} + 2\mathcal{R}_n,
\]
where $\mathcal{R}_n = \sum_{j=1}^m \abs{\sigma_j(\eta)}\cos(2\pi n\theta_{j} + \phi_j)$ for some $\phi_j\in\R$. Since $\eta\alpha^n\in\Z[\alpha]$, we have $\text{Tr}(\eta\alpha^n)\in\Z$. This implies
\[
\left\{\eta\alpha^n\right\} + 2\mathcal{R}_n \: (\text{mod }1) \longrightarrow 0, \text{ as } n\to\infty.
\]
This convergence implies the next fact (this is proved formally in Corollary \ref{cor}): let $J\subseteq[0,1]$ be an interval. Then
\[
\lim_{N\to\infty} \dfrac{\#\left\{n\leq N \middle| \left\{\eta\alpha^n\right\}\in J\right\}}{N} =  \int_{(\R/\Z)^m} \mathds{1}_{_J} \left(-2\sum_{j=1}^m \abs{\sigma_j(\eta)}\cos(x_j)\right)d\vec{x}.
\]

The last equality leave us the problem of understanding the integral of the indicator function of some interval. This is the same strategy used in \cite{akiyama2004salem} to prove Theorem \ref{akiyama}. The extra difficulty in our case is that we have to manage the parameters $\abs{\sigma_j(\eta)}$, whereas in \cite{akiyama2004salem} they only work the case $\eta = 1$, which implies $\abs{\sigma_j(\eta)} = 1$ for all $j=1,\dots,m$.

The solution is proving a much weaker inequality (in fact it is not possible to obtain the same one according to Proposition \ref{SalemProp}): we will see in the next section that for some a suitable $\delta > 0$, if we consider the interval $J(\delta) = [\delta,1-\delta]$, then
\begin{equation}\label{intBound}
\int_{(\R/\Z)^m} \mathds{1}_{_J(\delta)} \left(-2\sum_{j=1}^m \abs{\sigma_j(\eta)}\cos(x_j)\right)d\vec{x} \geq 1/2.
\end{equation}

This is the technical part of the proof, but the strategy is the same as in \cite{akiyama2004salem}: we use the family of Selberg polynomials defined in subsection \ref{polynomials} to approximate the indicator function of an interval. To manage the coefficients $\abs{\sigma_j(\eta)}$ we follow the strategy used in \cite{bufetov2014modulus}, using an inequality due to Garsia (Proposition \ref{Garsia}).
This is done in Lemma \ref{lemma1}. The proofs of Lemmas \ref{lemma2} and \ref{lemma4} are very similar to the former one, and we leave the details to the Appendix.

Finally, once we have proved (\ref{intBound}), the proof will be straightforward: an application of Lemma \ref{lemmaprod} and Proposition \ref{Spectralbound} will yield the desired modulus of continuty for the spectral measure.

\section{Proof of Theorem \ref{theorem1} and Theorem \ref{theorem2}}
In this section we give the proof of Theorem \ref{theorem1}. We begin by fixing the notation and hypothesis for the rest of this section. Let $\alpha$ be a Salem number of degree $d=2m+2$ and $\alpha_1 =\sigma_1(\alpha)= e^{2\pi i \theta_1}, \dots, \alpha_m =\sigma_m(\alpha)= e^{2\pi i \theta_m}$ the Galois conjugates on the upper half of $S^1$, and $\sigma_1,\dots,\sigma_m$ the respective embeddings. Let $\sigma_0: \Q(\alpha) \hookrightarrow \C$ be the embedding corresponding to $\sigma_0(\alpha)=\alpha^{-1}$. Let $\eta = \dfrac{1}{L}(l_0 + \dots + l_{d-1}\alpha^{d-1}) \in\Q(\alpha)\setminus\{0\}$, with $l_1,\dots,l_{d-1} \in \Z$ and $L \in \N$. We will denote $\vec{x}=(x_1,\dots,x_m)\in(\R/\Z)^m$.

\subsection{Preliminary results}
We begin with a proposition from \cite{dubickas2005there} (see also \cite{zaimi2006integer,dubickas2001sequences}).
\begin{prop}\label{propDub}
	Consider
	\begin{align*}
	U(z) &= l_0 + \dots + l_{d-1}\cos(2\pi(d-1)z),\\
	V(z) &= l_1\sin(2\pi z) + \dots + l_{d-1}\sin(2\pi(d-1)z),\\
	\phi(z) &= \arctan(V(z)/U(z)).
	\end{align*}
	Define $\mathcal{R}_n = \sum_{j=1}^m \sqrt{U^2(\theta_j)+V^2(\theta_j)}\cos(2\pi n\theta_j - \phi(\theta_j))$. There exists a positive integer $P \leq L^d$ such that for every $p\in\left\{0,\dots,P-1\right\}$ there exists
	$a_p\in\left\{0,\dots,L-1\right\}$ satisfying
	\[
	\left\{\eta\alpha^{Pn+p}\right\} + \dfrac{2\mathcal{R}_{Pn+p}}{L} \: (\textnormal{mod }1) \longrightarrow \dfrac{a_p}{L} \quad \text{as } n\to \infty.
	\]
\end{prop}
\begin{proof}
	We only sketch the proof, more details may be found in \cite{dubickas2005there}, Section 4. Note that $U(\theta_{j})=\Re(\sigma_j(L\eta))$ and $V(\theta_{j})=\Im(\sigma_j(L\eta))$ and so $\sqrt{U^2(\theta_j)+V^2(\theta_j)} = \abs{\sigma_j(L\eta)}$, for $j = 1,\dots,m$. By definition of the trace, for every $n\geq 0$
	\begin{equation}\label{equation}
	\text{Tr}(L\eta\alpha^n) = L\eta\alpha^n + \sigma_0(L\eta)\alpha^{-n} + 2\mathcal{R}_n.
	\end{equation}
	
	The sequence $(\text{Tr}(L\eta\alpha^n)\: (\text{mod }L))_{n\geq 0}$ is periodic of period $P \leq L^d$: let $X^d - c_{d-1}X^{d-1} -\dots- c_1X - 1$ be the minimal polynomial of $\alpha$. By definition, the sequence $(\text{Tr}(L\eta\alpha^n))_{n\geq 0} = (\mathfrak{A}_n)_{n\geq0}$ satisfies the linear recurrence
	\[
	\mathfrak{A}_{n+d} = c_{d-1}\mathfrak{A}_{n+d-1} -\dots- c_1\mathfrak{A}_{n+1} - \mathfrak{A}_{n}.
	\]
	By the pigeonhole principle, $(\text{Tr}(L\eta\alpha^n)\: (\text{mod }L))_{n\geq 0}$ is ultimately periodic of period $P \leq L^d$. To see that it is periodic we argue by contradiction. Suppose the sequence $(\text{Tr}(L\eta\alpha^n)\: (\text{mod }L))_{n\geq k}$ is periodic, and $k>0$ is minimal. Then, $\mathfrak{A}_{k+d-1} \equiv \mathfrak{A}_{k+d-1+P} \: (\text{mod }L)$ and since
	\begin{align*}
		\mathfrak{A}_{k+d-1} &= c_{d-1}\mathfrak{A}_{k+d-2} -\dots- c_1\mathfrak{A}_{k} - \mathfrak{A}_{k-1},\\
		\mathfrak{A}_{k+d-1+P} &= c_{d-1}\mathfrak{A}_{k+d-2+P} -\dots- c_1\mathfrak{A}_{k+P} - \mathfrak{A}_{k-1+P},
	\end{align*}
	we get $\mathfrak{A}_{k-1} \equiv \mathfrak{A}_{k-1+P} \: (\text{mod }L)$, a contradiction. 
		
	Rearranging (\ref{equation}) we conclude that
	\[
	\left\{\eta\alpha^n\right\} + \dfrac{2\mathcal{R}_n}{L} \: (\text{mod }1) \longrightarrow \dfrac{a_p}{L} \text{ as } n\to\infty,
	\]
	where $a_p\in\left\{0,\dots,L-1\right\}$ and $n\equiv p \: (\text{mod }P)$.
\end{proof}
\begin{cor}\label{cor} Let $J\subseteq[0,1]$ be an interval and set
	\[
	\mathcal{R}(x_1,\dots,x_m) = \sum_{j=1}^m \sqrt{U^2(\theta_j)+V^2(\theta_j)}\cos(2\pi x_j).
	\]
	Then
	\[
	\lim_{N\to\infty} \dfrac{\#\left\{n\leq N \middle| \left\{\eta\alpha^n\right\}\in J\right\}}{N} =  \dfrac{1}{P} \sum_{p=0}^{P-1} \int_{(\R/\Z)^m} \mathds{1}_{_J} \left(\dfrac{-2\mathcal{R}(\vec{x})+a_j}{L}\right)d\vec{x}.
	\]
\end{cor}
\begin{proof}
	By Proposition \ref{ratind}, the numbers $1,\theta_1,\dots,\theta_m$ are rationally independent, then the uniform distribution $\text{mod } \Z^m$ of $((n\theta_1,\dots,n\theta_m))_{n\geq1}$ justifies the last of the next equalities. We can ignore the term $\sigma_0(\eta)\alpha^{-n}$ in (\ref{ignore}) since this term goes to zero and the sequences $z^{(p)}_n = \dfrac{-2\mathcal{R}_{Pn+p}}{L} + \dfrac{a_p}{L}$ have continuous asymptotic distribution functions.
	
	\small{
		\begin{align}
		&\lim_{N\to\infty} \dfrac{\#\left\{n\leq N \middle| \left\{\eta\alpha^n\right\}\in J\right\}}{N}\\
		&=\lim_{N\to\infty}  \sum_{p=0}^{P-1}\dfrac{ \#\left\{n\leq N,\: n\equiv p\: (\text{mod }P) \middle| \left\{\eta\alpha^n\right\}\in J \right\} }{N} \\
		&= \sum_{p=0}^{P-1} \lim_{N\to\infty} \dfrac{ \#\left\{n\leq N,\: n\equiv p\: (\text{mod }P) \middle| \dfrac{-2\mathcal{R}_n}{L} + \dfrac{a_p}{L} - \sigma_0(\eta)\alpha^{-n}\: (\text{mod }1)\in J \right\} }{N} \label{ignore} \\
		&= \dfrac{1}{P} \sum_{p=0}^{P-1} \lim_{N\to\infty} \dfrac{\#\left\{n\leq N,\: n\equiv p\: (\text{mod }P) \middle| \dfrac{-2\mathcal{R}_n}{L} + \dfrac{a_p}{L}\:(\text{mod }1) \in J \right\} }{N/P}\\
		&= \dfrac{1}{P} \sum_{p=0}^{P-1} \int_{(\R/\Z)^m} \mathds{1}_{_J} \left(\dfrac{-2\mathcal{R}(\vec{x})+a_p}{L}\right)d\vec{x}.
		\end{align}}
\end{proof}
For a fixed $1/2>\delta>0$, denote by $J(\delta)$ the interval $[\delta,1-\delta]\subseteq[0,1]$. Denote $H_j = \sqrt{U^2(\theta_j)+V^2(\theta_j)}$ and $H = \sum_{j=1}^m H_j = \max_{\vec{x}\in(\R/\Z)^m} \left|\mathcal{R}(\vec{x})\right|$. The next lemmas show that all integrals in the last corollary are greater than a fixed positive constant for a suitable choice of $\delta$. We will divide the analysis depending on the size of $2H/L$ with respect to $\delta_1(\alpha):=1/\mathscr{L}(\alpha)$, where $\mathscr{L}(\alpha)$ denotes the sum of the absolute values of the coefficients of the minimal polynomial of $\alpha$.

\begin{lemma}\label{lemma1}
	Let $\beta = \alpha^P$ and $\,\tilde{\eta} = \eta\alpha^p$ for some $p\in\left\{0,\dots,P-1\right\}$, in order to have $\eta\alpha^{Pn+p} = \tilde{\eta}\beta^n$ for all $n\geq0$. Suppose $\text{Tr}(L\tilde{\eta}\beta^n) \equiv 0 \:(\text{mod }L)$ for all $n\geq 0$ and $2H/L < \delta_1(\beta)/2$. Then there exists $\delta = \delta(L,\abs{\eta},\abs{\sigma_0(\eta)})>0$ such that
	\[
	\int_{(\R/\Z)^m} \mathds{1}_{_{J(\delta)}}\left(\dfrac{-2\mathcal{R}(\vec{x})}{L}\right)d\vec{x} \geq 1/2.
	\]
\end{lemma}
\begin{remark}
	From the proof of Lemma \ref{lemma1} it is easy to deduce that in the dependence of $\delta = \delta(L,\abs{\eta},\abs{\sigma_0(\eta)})$ what is important is an upper bound of $L$ and $\abs{\sigma_0(\eta)}$ and both bounds for $\abs{\eta}$.
\end{remark}
\begin{proof}
	Note that by definition $\beta$ is also a Salem number of degree $d$, fact that is used implicitly. Let us follow \cite{bufetov2014modulus} and write
	\[
	\tilde{\eta}\beta^n = K_n(\tilde{\eta},\beta) + \epsilon_n(\tilde{\eta},\beta), \quad K_n\in\Z,\: -1/2<\epsilon_n\leq1/2.
	\]
	Let $n_0$ be the smallest nonnegative integer such that $\abs{\tilde{\eta}}\beta^{n_0}\geq 1$.
	Equation (\ref{equation}) and $2H/L < \delta_1(\beta)/2$ imply $\abs{\epsilon_n} < \delta_1(\beta)$ for all $n\geq n_1$, where $n_1= \ceil{\log_{\beta}(2\abs{\sigma_0(\tilde{\eta})}/\delta_1(\beta)}$. Indeed,
	$\abs{\epsilon_n} = \min(\{\tilde{\eta}\beta^n\},1-\{\tilde{\eta}\beta^n\})$. Then, for all $n\geq n_1$, we have
	\[
	\abs{\epsilon_n} = \left|\left|-\sigma_0(\tilde{\eta})\beta^{-n} - \dfrac{2\mathcal{R}_{Pn+p}}{L}\right|\right|_{\R/\Z} \leq \abs{\sigma_0(\tilde{\eta})}\beta^{-n} + \left|\dfrac{2\mathcal{R}_{Pn+p}}{L}\right| < \delta_1(\beta).
	\]
	
	Denote $n_2 = n_2(\abs{\tilde{\eta}},\abs{\sigma_0(\tilde{\eta})}) = \max(n_0,n_1)$ and define also the vectors
	\[
	\vec{\epsilon}_{n} = \begin{pmatrix}
	\epsilon_{n}\\
	\epsilon_{n+1}\\
	\vdots\\
	\epsilon_{n+d-2}\\
	\epsilon_{n+d-1}
	\end{pmatrix}
	,\quad
	\vec{K}_{n} = \begin{pmatrix}
	K_{n}\\
	K_{n+1}\\
	\vdots\\
	K_{n+d-2}\\
	K_{n+d-1}
	\end{pmatrix}.
	\]
	As shown in \cite{bufetov2014modulus}, we may prove that if $\abs{\epsilon_{n_2+n}} < \delta_1(\beta)$ for all $n\geq 0$, then $\vec{\epsilon}_{n_2+n} = \mathfrak{C}(\beta)^{n}\vec{\epsilon}_{n_2}$ for all $n\geq 0$, where $\mathfrak{C}(\beta)$ is the companion matrix of the minimal polynomial of $\beta$: let $X^d - c_{d-1}X^{d-1} - \dots - c_0$ be the minimal polynomial of $\beta$, then
	\[
	\mathfrak{C}(\beta) = \begin{pmatrix}
	0 & 1 & 0 & \cdots & 0\\
	0 & 0 & 1 & \cdots & 0\\
	\vdots & \vdots & \vdots & \ddots & \vdots\\
	0 & 0 & 0 & \cdots & 1\\
	c_0 & c_1 & c_2 & \cdots & c_{d-1}
	\end{pmatrix}.
	\]

	Let $\vec{E_1},\vec{E_2},\vec{e_1},\dots,\vec{e_m},\overline{\vec{e_1}},\dots,\overline{\vec{e_m}} $ and $\vec{E}_1^{\:*},\vec{E}_2^{\:*},\vec{e}_1^{\:*},\dots,\vec{e}_m^{\:*},\overline{\vec{e}_1^{\:*}},\dots,\overline{\vec{e}_m^{\:*}} $ be the eigenbasis and dual basis of $\mathfrak{C}(\beta)$ respectively. The vectors $\vec{e}_j$, $\vec{e}_j^{\:*}$ are explicitly given by
	\[
	\vec{e}_{j} = \begin{pmatrix}
	1\\
	\beta_j\\
	\vdots\\
	\beta_j^{d-2}\\
	\beta_j^{d-1}
	\end{pmatrix}
	,\quad
	\vec{e}_j^{\:*} = \begin{pmatrix}
	c_0\beta_j^{d-2}\\
	c_1\beta_j^{d-2}+c_0\beta_j^{d-3}\\
	\vdots\\
	c_{d-2}\beta_j^{d-2}+\dots+c_1\beta_j+ c_0\\
	\beta_j^{d-1}
	\end{pmatrix},
	\]
	where each $\beta_j = \sigma_j(\beta)$, for $j=1,\dots,m$ is a Galois conjugate of $\beta$.
	
	Decompose $\vec{\epsilon}_{n_2} = B_0(\tilde{\eta},n_2)\vec{E_1} + B_1(\tilde{\eta},n_2)\vec{E_2} + \sum_{j=1}^m \left( b_j(\tilde{\eta},n_2)\vec{e_j} + \overline{b_j(\tilde{\eta},n_2)}\,\overline{\vec{e}_j}\right)$, where $\vec{E_1},\vec{E_2}$ are associated to $\beta,\beta^{-1}$ respectively.  In fact, we can show that $\vec{\epsilon}_{n_2+n} = \mathfrak{C}(\beta)^n\vec{\epsilon}_{n_2}$ for all $n\geq 0$ implies $B_0(\tilde{\eta},n_2) = 0$: indeed, since $B_0(\tilde{\eta},n_2) = {\langle\vec{\epsilon}_{n_2},\vec{e}_1^{\:*}\rangle}/{\langle\vec{e}_1,\vec{e}_1^{\:*}\rangle}$, for all $n\geq0$ we have
	\begin{align*}
	\beta^n\abs{\langle\vec{\epsilon}_{n_2},\vec{e}_1^{\:*}\rangle} &= \abs{\langle\mathfrak{C}(\beta)^n\vec{\epsilon}_{n_2},\vec{e}_1^{\:*}\rangle}\\
	&\leq \abs{\langle\vec{\epsilon}_{n_2+n},\vec{e}_1^{\:*}\rangle}\\
	&\leq \norm{\vec{\epsilon}_{n_2+n}}_2\norm{\vec{e}_1^{\:*}}_2\\
	&\leq \sqrt{m}\delta_1\norm{\vec{e}_1^{\:*}}_2,
	\end{align*}
	which obviously leads to a contradiction as $n$ goes to infinity if $B_0(\tilde{\eta},n_2)\neq0$.
	
	The equation for $\vec{\epsilon}_{n_2+n}$ may be written coordinatewise as
	\[
	\epsilon_{n_2+n} = B_1(\tilde{\eta},n_2)\beta^{-n} + 2\sum_{j=1}^m\abs{b_j(\tilde{\eta},n_2)}\cos(2\pi n\widetilde{\theta_{j}} - \phi(\widetilde{\theta_{j}})),
	\]
	where $\widetilde{\theta_{j}} = P\theta_{j}$. From this decomposition and the fact that
	\[
	\epsilon_n \,(\text{mod } 1)=  -\sigma_0(\tilde{\eta})\beta^{-n} - \dfrac{2\mathcal{R}_{Pn+p}}{L} \,(\text{mod } 1),
	\]
	according to (\ref{equation}), we deduce that $H_j/L = \abs{b_j(\tilde{\eta},n_2)}$, for $j=1,\dots, m$.
	
	Denote $z(\vec{x}) = 2\sum_{j=1}^m \abs{b_j(\tilde{\eta},n_2)}\cos(2\pi x_j)$. Set $\mathfrak{a}=\delta$, $\mathfrak{b}=1-\delta$ for $\delta>0$ a parameter we will choose later suitably to satisfy the conclusion of the lemma. We face now the task to establish bounds for the integrals of Selberg polynomials, and we do it in the same manner as in \cite{akiyama2004salem}. In the calculations below we write $z = z(\vec{x})$, $b_j = b_j(\tilde{\eta},n_2)$ and all integrals are over $(\R/\Z)^m$. Let us start by the Beurling polynomial of degree $N$ (for the Vaaler polynomial write $\sin(x) = \Im(e^{ix})$ and use Proposition \ref{Bessel}).
	\begin{align*}
	\int \mathcal{B}_N(-z+\mathfrak{a})d\vec{x} &= \int \mathcal{V}_N(-z+\mathfrak{a}) + \dfrac{1}{2(N+1)}\Delta_{N+1}(-z+\mathfrak{a})d\vec{x} \\
	&= \int \dfrac{1}{N+1} \sum_{k=1}^N f\left(\dfrac{k}{N+1}\right)\sin(2\pi k(-z+\mathfrak{a}))d\vec{x}\\
	& + \: \int \dfrac{1}{2(N+1)} \left\{ 1 + \sum_{0<\abs{k}\leq N+1} \left(1 - \dfrac{\abs{k}}{N+1}\right)e^{2\pi i k (-z+\mathfrak{a})} \right\}d\vec{x} \\
	&= \underbrace{ \dfrac{-1}{N+1} \sum_{k=1}^N f\left(\dfrac{k}{N+1}\right)\sin(2\pi k\mathfrak{a}) \prod_{j=1}^m J_0(4\pi k\abs{b_{j}}) }_{= \,(1)} \\
	& + \: \underbrace{ \dfrac{1}{2(N+1)} \left\{ 1 + \sum_{0<\abs{k}\leq N+1} \left(1 - \dfrac{\abs{k}}{N+1}\right)e^{2\pi i k\mathfrak{a}} \prod_{j=1}^m J_0(4\pi k\abs{b_{j}}) \right\} }_{= \,(2)},
	\end{align*}
	where we used Proposition \ref{IntegralBeurling} in the last equality. Let us denote $\mathfrak{g} = \left(\prod_{j=1}^m \abs{b_j}\right)^{1/m}$ and define for each nonnegative integer $l$ the number $k_l = \floor{\mathfrak{g}^{-1}l}$. By the AM-GM inequality,
	\[
	\mathfrak{g} \leq \dfrac{1}{m}\sum_{j=1}^{m} \abs{b_j} = \dfrac{1}{m}\dfrac{H}{L}<\dfrac{\delta_1(\beta)}{4m}< 1.
	\]
	
	From now on, we will consider $N$ of the form $\floor{\mathfrak{g}^{-1}T}$, for each $T\in\N$. We use below the inequality for the Bessel function from Proposition \ref{Bessel}.
	\begin{align*}
	\abs{(2)} &\leq \dfrac{1}{2(N+1)} \left\{ 1 + 2\sum_{k=1}^{N+1} \left(1 - \dfrac{k}{N+1}\right) \prod_{j=1}^m \abs{J_0(4\pi k\abs{b_{j}})} \right\} \\
	&\leq \dfrac{1}{2(N+1)} \left\{ 1 + 2\sum_{l=0}^{T-1}\sum_{k=k_l+1}^{k_{l+1}} \left(1 - \dfrac{k}{N+1}\right) \prod_{j=1}^m\abs{J_0(4\pi k\abs{b_{j}})} \right\} \\
	&\leq \dfrac{1}{2(N+1)} \left\{ 1 + 2\sum_{k=1}^{k_{1}} \left(1 - \dfrac{k}{N+1}\right) \prod_{j=1}^m \underbrace{\abs{J_0(4\pi k\abs{b_{j}})}}_{\leq 1} \right.\\
	& \left.+\: 2\sum_{l=1}^{T-1}\sum_{k=k_l+1}^{k_{l+1}} \left(1 - \dfrac{k}{N+1}\right) \prod_{j=1}^m\underbrace{\abs{J_0(4\pi k\abs{b_{j}})}}_{\leq 1/\sqrt{2}\pi(k\abs{b_j})^{1/2}} \right\}\\
	&\leq \dfrac{1}{2(N+1)} \left\{ 1 +  2\mathfrak{g}^{-1} + \dfrac{2}{(\sqrt{2}\pi)^m}\sum_{l=1}^{T-1}\sum_{k=k_l+1}^{k_{l+1}} \underbrace{(k\mathfrak{g})^{-m/2}}_{\leq \:l^{-m/2}} \right\} \\
	&\leq \dfrac{1}{2(N+1)} \left\{ 1 + 2\mathfrak{g}^{-1} + \dfrac{2}{(\sqrt{2}\pi)^m}\sum_{l=1}^{T-1} l^{-m/2} \underbrace{(k_{l+1}-k_l)}_{\leq \: \mathfrak{g}^{-1}} \right\} \\
	&\leq \dfrac{1}{2(N+1)} \left\{ 1 + 2\mathfrak{g}^{-1} + \dfrac{2\mathfrak{g}^{-1}}{(\sqrt{2}\pi)^m}O(\sqrt{T}) \right\} \\
	&= O(1/\sqrt{T}),
	\end{align*}
	where the constant implicit in the $O$ sign does not depend on $\tilde{\eta}$ nor $n_2$. Similarly,
	\[
	\int \mathcal{B}_N(z-\mathfrak{b})d\vec{x} = \dfrac{1}{N+1} \sum_{k=1}^N f\left(\dfrac{k}{N+1}\right)\sin(2\pi k\mathfrak{b}) \prod_{j=1}^m J_0(4\pi k\abs{b_{j}}) + O(1/\sqrt{T}).
	\]
	
	Having estimated the integral of the Beurling polynomials, we continue with the Selberg ones:
	\begin{gather*}
	\left|\int \mathcal{B}_N(z-\mathfrak{b}) + \mathcal{B}_N(-z+\mathfrak{a}) d\vec{x}\,\right| \\
	\leq \\
	\left| \dfrac{1}{N+1}  \sum_{k=1}^N f\left(\dfrac{k}{N+1}\right) (\sin(2\pi k\mathfrak{b}) - \sin(2\pi k\mathfrak{a})) \prod_{j=1}^{m} J_0(4\pi k\abs{b_j}) \right| + O(1/\sqrt{T}) \\
	= \\
	\underbrace{ \left| \dfrac{2}{N+1}  \sum_{k=1}^N f\left(\dfrac{k}{N+1}\right) \sin(2\pi k\delta) \prod_{j=1}^{m} J_0(4\pi k\abs{b_j}) \right| }_{=(3)} + \:O(1/\sqrt{T}).
	\end{gather*}
	
	Fix $\epsilon>0$ and let us take $\xi<1/2$ such that $\pi\xi/\sin(\pi\xi) \leq 1+\epsilon$ (recall the definition and property of $f$ in Definition \ref{defValeer}). Then, for $T$ big enough, we obtain
	\begin{align*}
	(3) &\leq \dfrac{2}{N+1} \sum_{k=1}^N \left|f\left(\dfrac{k}{N+1}\right)\sin(2\pi k\delta)\prod_{j=1}^m J_0(4\pi k\abs{b_{j}})\right| \\	
	&\leq \dfrac{2}{N+1}\left[\: \sum_{k=1}^{k_{1}} \left| f\left(\dfrac{k}{N+1}\right) \right| \underbrace{\left| \sin(2\pi k\delta) \right|}_{\leq2\pi k\delta} \prod_{j=1}^m \underbrace{\abs{J_0(4\pi k\abs{b_{j}})}}_{\leq 1} \right.\\
	&\quad + \sum_{l=1}^{\floor{\xi T}}\sum_{k=k_l+1}^{k_{l+1}} \underbrace{\left|f\left(\dfrac{k}{N+1}\right)\right| }_{\leq \left|f(l\mathfrak{g}^{-1}/(N+1))\right|} \underbrace{\abs{\sin(2\pi k\delta)}}_{\leq 1} \prod_{j=1}^m\underbrace{\abs{J_0(4\pi k\abs{b_{j}})}}_{\leq 1/\sqrt{2}\pi(k\abs{b_{j}})^{1/2}} \\
	&\quad \left. + \sum_{l=\floor{\xi T}+1}^{T-1}\:\sum_{k=k_l+1}^{k_{l+1}} \left|f\left(\dfrac{k}{N+1}\right)\right| \underbrace{\abs{\sin(2\pi k\delta)}}_{\leq 1} \prod_{j=1}^m \underbrace{\abs{J_0(4\pi k\abs{b_{j}})}}_{\leq 1/\sqrt{2}\pi(k\abs{b_{j}})^{1/2}} \: \right]\\
	&\leq  \dfrac{4\pi\delta}{N+1} \sum_{k=1}^{k_{1}} \left(\dfrac{\pi\xi}{\sin(\pi\xi)}\dfrac{N+1}{\pi k}+\dfrac{1}{\pi}\right)k \\
	&\quad + \dfrac{2}{(\sqrt{2}\pi)^m(N+1)}  \sum_{l=1}^{\floor{\xi T}} \left|f\left(\dfrac{l\mathfrak{g}^{-1}}{N+1}\right)\right| \sum_{k=k_l+1}^{k_{l+1}} \underbrace{(k\mathfrak{g})^{-m/2}}_{\leq \:l^{-m/2}} \\
	&\quad + \dfrac{2}{(\sqrt{2}\pi)^m(N+1)} \sum_{l=\floor{\xi T}+1}^{T-1} \left( \dfrac{1-\xi}{\sin(\pi(1-\xi))} + \dfrac{1}{\pi}\right) \:\sum_{k=k_l+1}^{k_{l+1}} \underbrace{(k\mathfrak{g})^{-m/2}}_{\leq \:l^{-m/2}}\\
	&\leq 4(1+\epsilon)\delta \mathfrak{g}^{-1} + O(1/T) + \dfrac{2\mathfrak{g}^{-1}}{(\sqrt{2}\pi)^m(N+1)}\sum_{l=1}^{\floor{\xi T}} \left(\dfrac{\pi\xi}{\sin(\pi\xi)}\dfrac{(N+1)\mathfrak{g}}{\pi l}+\dfrac{1}{\pi}\right)l^{-m/2} \\
	&\quad + O(1/\sqrt{T}) \\
	&\leq 4(1+\epsilon)\delta \mathfrak{g}^{-1} + \dfrac{2(1+\epsilon)\zeta\left(\dfrac{m+2}{2}\right)}{\pi(\sqrt{2}\pi)^m}  + O(1/\sqrt{T}),
	\end{align*}
	where $\zeta$ denotes the Riemann zeta function. In the same manner,
	\[
	\left|\int \mathcal{B}_N(-z+\mathfrak{b}) + \mathcal{B}_N(z-\mathfrak{a}) d\vec{x}\,\right| \leq 4(1+\epsilon)\delta \mathfrak{g}^{-1} + \dfrac{2(1+\epsilon)\zeta\left(\dfrac{m+2}{2}\right)}{\pi(\sqrt{2}\pi)^m} + O(1/\sqrt{T}).
	\]
	
	It may be checked that $\dfrac{2\zeta\left(\dfrac{m+2}{2}\right)}{\pi(\sqrt{2}\pi)^m} \leq \dfrac{\sqrt{2}\zeta(3/2)}{\pi^2} < 0.4$ for all $m\geq1$. Since $\epsilon>0$ was arbitrarily chosen, by taking $T\to\infty$ we conclude by means of (\ref{propertySelberg}) that
	\[
	\left| \int_{(\R/\Z)^m} \mathds{1}_{_{J(\delta)}}(z(\vec{x}))d\vec{x} - \abs{J(\delta)}\right| \leq  4\delta \mathfrak{g}^{-1} + 0.4,
	\]
	which implies 
	\[
	1 - 2\delta - 0.4 - 4\delta \mathfrak{g}^{-1} \leq \int_{(\R/\Z)^m} \mathds{1}_{_{J(\delta)}}(z(\vec{x}))d\vec{x}.
	\]
	
	To estimate $\mathfrak{g}$, notice that $\abs{b_j} = \abs{ \langle\vec{K}_{n_2},\vec{e}_j^{\:*}\rangle }/\abs{\langle\vec{e}_j,\vec{e}_j^{\:*}\rangle}$. The numerator of this fraction is a polynomial with integer coefficients of degree $d-1$ evaluated at $\beta_j$ (in particular non zero, since $\beta_j$ is of degree $d$), height at most $K_{n_2}$ (except for multiplication by a constant only depending on $\beta$) and the denominator only depends on $\beta$. From Lemma \ref{Garsia}, we obtain $\mathfrak{g} \geq c\abs{\tilde{\eta}}^{-d}\beta^{-n_2d}$ for some explicit constant $c = c(\beta)$ depending on $\beta = \alpha^P$ and, by the bound on the period, it may be changed to a dependence on $\alpha$ and $L$. 
	
	We may solve the inequality $1/2 < 1 - 2\delta - 0.4 - 4\delta c^{-1}\abs{\tilde{\eta}\beta^{n_2}}^{d}$ for some $\delta = \delta(\abs{\tilde{\eta}},\abs{\sigma_0(\tilde{\eta})})$, which will satisfy the conclusion. In fact, just notice that by uniform distribution (as in Corollary \ref{cor}), 
	\begin{align*}
	\int_{(\R/\Z)^m} &\mathds{1}_{_{J(\delta)}}(z(\vec{x}))d\vec{x} \\
	&= \lim_{N\to\infty} \dfrac{\#\left\{n\leq N\, \middle|\, \abs{\epsilon_{n}(\tilde{\eta},\beta)}>\delta\right\}}{N} \\
	&= \lim_{N\to\infty} \dfrac{1}{N/P}\#\left\{n\leq N,\: n\equiv p\: (\text{mod }P) \middle|\,\left|\left|\dfrac{-2\mathcal{R}_n}{L}\right|\right|_{\R/\Z} >\delta \right\} \\
	&= \lim_{N\to\infty} \dfrac{1}{N/P}\#\left\{n\leq N,\: n\equiv p\: (\text{mod }P) \middle|\,\dfrac{-2\mathcal{R}_n}{L} \:(\text{mod }1) \in J(\delta) \right\} \\
	&= \int_{(\R/\Z)^m} \mathds{1}_{_{J(\delta)}}\left(\dfrac{-2\mathcal{R}(\vec{x})}{L}\right)d\vec{x}.
	\end{align*}
	
	The dependence of $\delta$ on $(\abs{\tilde{\eta}},\abs{\sigma_0(\tilde{\eta})})$ may be changed to a dependence on $(L,\abs{\eta},$ $\abs{\sigma_0(\eta)})$, since $\abs{\eta} \leq \abs{\tilde{\eta}} \leq \abs{\eta}\alpha^P \leq \abs{\eta}\alpha^{L^d}$ and $\abs{\sigma_0(\tilde{\eta})} \leq \abs{\sigma_0(\eta)}$.
\end{proof}
\begin{lemma}\label{lemma2}
	Let $\beta$ and $\tilde{\eta}$ be as in Lemma \ref{lemma1}. Suppose $\text{Tr}(L\tilde{\eta}\beta^n) \equiv 0 \:(\text{mod }L)$ for all $n\geq 0$ and $2H/L \geq \delta_1(\beta)/2$. Then there exists $\delta = \delta(L)>0$ such that
	\[
	\int_{(\R/\Z)^m} \mathds{1}_{_{J(\delta)}}\left(\dfrac{-2\mathcal{R}(\vec{x})}{L}\right)d\vec{x} \geq 1/2.
	\]
\end{lemma}
\begin{proof}
	The calculations are analogous. The details are found in the Appendix.
\end{proof}
\begin{lemma}\label{lemma3}
	Let $\beta$ and $\tilde{\eta}$ be as in Lemma \ref{lemma1}. Suppose $\text{Tr}(L\tilde{\eta}\beta^n) \equiv l \neq 0 \:(\text{mod }L)$ for all $n\geq 0$ and $2H<1/2$. Then
	\[
	\int_{(\R/\Z)^m} \mathds{1}_{_{J(1/4L)}}\left(\dfrac{-2\mathcal{R}(\vec{x}) + l}{L}\right)d\vec{x} \geq 1/2.
	\]
\end{lemma}
\begin{proof}
	By definition of $\mathcal{R}(\vec{x})$, since $2H < 1/2$ (recall $H = \max_{\vec{x}\in(\R/\Z)^m} \left|\mathcal{R}(\vec{x})\right|$) we can deduce
	\[
	\dfrac{1}{4L} \leq \dfrac{-1/2 + 1}{L} \leq \dfrac{-2\mathcal{R}(\vec{x}) + l}{L} \leq \dfrac{1/2 + L-1}{L} \leq 1-\dfrac{1}{4L},
	\]
	i.e.,
	\[
	\dfrac{-2\mathcal{R}(\vec{x}) + l}{L} \in J(1/4L),
	\]
for all $\vec{x} \in (\R/\Z)^m$. That is,
	\[
	\int_{(\R/\Z)^m} \mathds{1}_{_{J(1/4L)}}\left(\dfrac{-2\mathcal{R}(\vec{x}) + l}{L}\right)d\vec{x}=1 \geq 1/2.
	\]
\end{proof}
\begin{lemma}\label{lemma4}
	Let $\beta$ and $\tilde{\eta}$ be as in Lemma \ref{lemma1}. Suppose $\text{Tr}(L\tilde{\eta}\beta^n) \equiv l \neq 0 \:(\text{mod }L)$ for all $n\geq 0$ and $2H\geq1/2$. Then there exists a $\delta=\delta(L)>0$ such that
	\[
	\int_{(\R/\Z)^m} \mathds{1}_{_{J(\delta)}}\left(\dfrac{-2\mathcal{R}(\vec{x}) + l}{L}\right)d\vec{x} \geq 1/2.
	\]
\end{lemma}
\begin{proof}
	The calculations are analogous. The details are found in the Appendix.
\end{proof}

\subsection{Conclusion}
\textit{Proof of Theorem \ref{theorem1}}: let $\kappa\in\Z[\alpha]$ be the positive constant appearing in Lemma \ref{lemmaprod} and fix $A,B,C>1$ for the rest of this section. Define
\[
B_{\kappa} = \kappa B, \quad C_{\kappa} = \abs{\sigma_0(\kappa)}C.
\]
According to Lemmas \ref{lemma1}, \ref{lemma2}, \ref{lemma3} and \ref{lemma4} above (see also the remark below Lemma \ref{lemma1}), we can find an explicit $\delta = \delta(A,B_\kappa,C_\kappa)>0$ such that 
\[
\int_{(\R/\Z)^m} \mathds{1}_{_{J(\delta)}}\left(\dfrac{-2\mathcal{R}(\vec{x}) + a_p}{L}\right)d\vec{x} \geq 1/2,
\]
for all $p=0,\dots,P-1$. In particular, by Corollary \ref{cor} there exists $N_0\geq 1$ such that for all $N\geq N_0$, $\abs{\eta}\in[B_{\kappa}^{-1},B_\kappa]$, $\abs{\sigma_0(\eta)}\leq C_\kappa$ and $L\leq A$:
\[
\dfrac{\#\left\{n\leq N \middle| \left\{\eta\alpha^n\right\}\in J(\delta) \right\}}{N} \geq 1/3.
\]

Note that if $\eta=\omega\kappa$, then $\abs{\omega} \in \Q(\alpha)\cap[B^{-1},B]$ and $\abs{\sigma_0(\omega)}\leq C$. Assume also $\eta$ is in reduced form and its denominator $L$ is less or equal to $A$, which implies the denominator of $\omega$ in reduced form is less or equal to $A$. By Lemma \ref{lemmaprod}, considering $R\geq R_0:= \alpha^{N_0 + 1}$ we obtain
\begin{align*}
\sup_{(x,s)\in\mathfrak{X}^{\vec{p}}_{\zeta}} \abs{S^f_R((x,s),\omega)} &\leq C_1 R (1-\lambda\delta^2)^{\log_{\alpha}(R)/3} \\
&= C_1 R(\alpha^{\log_\alpha(1-\lambda\delta^2)})^{\log_\alpha(R)/3} \\
&= C_1 R^{1+ \log_{\alpha}(1-\lambda\delta^2)/3} \\
&= C_1 R^{\tilde{\gamma}}, \quad \tilde{\gamma}=\tilde{\gamma}(A,B,C) \in (0,1).
\end{align*}

A modulus of continuity for the correlation measure $\nu_a$ ($a\in\A$) may be derived from the last inequality using Proposition \ref{Spectralbound}: for $\tilde{\gamma} = \tilde{\gamma}(A,B,C) = 1+ \log_{\alpha}(1-\lambda\delta^2)/3$ we have by Proposition \ref{Spectralbound}
\[
\nu_a([\omega-r,\omega+r]) \leq \pi^2C_1 r^{2(1-\tilde{\gamma})}=: cr^{\gamma},
\]
for all $0< r\leq r_0$ for some $r_0>0$ and $c>0$, the last one only depending on the substitution. We have proved Theorem \ref{theorem1}.

\textit{Proof of Theorem \ref{theorem2}}: in view of the uniform dependence of $\delta$ on the variables $B,C$ in Lemmas \ref{lemma3} and \ref{lemma4}, similar calculations as the ones we did before lead to an exponent $\gamma = \gamma(A)>0$ only depending on $A$: Corollary \ref{cor} and the existence of some $a_{p^*}\neq0$ (since $\text{Tr}(L\kappa\omega\alpha^n)\: \not\equiv 0\: (\text{mod }L)$ for some $n\in\N$ by hypothesis), yields
\begin{align*}
\lim_{N\to\infty} \dfrac{\#\left\{n\leq N \middle| \left\{\eta\alpha^n\right\}\in J(\delta) \right\}}{N} &= \dfrac{1}{P} \sum_{p=0}^{P-1} \int_{(\R/\Z)^m} \mathds{1}_{_{J(\delta)}}\left(\dfrac{-2\mathcal{R}(\vec{x}) + a_p}{L}\right)d\vec{x}\\
&\geq \dfrac{1}{P} \int_{(\R/\Z)^m} \mathds{1}_{_{J(\delta)}}\left(\dfrac{-2\mathcal{R}(\vec{x}) + a_{p^*}}{L}\right)d\vec{x}\\
&\geq \dfrac{1}{2A^d}.
\end{align*}
This means there exists $N_0\geq 1$ such that for all $N\geq N_0$
\[
\dfrac{\#\left\{n\leq N \middle| \left\{\eta\alpha^n\right\}\in J(\delta) \right\}}{N} \geq \dfrac{1}{3A^d}.
\]
Proceeding in the same way as before, we arrive to the conclusion of Theorem \ref{theorem2}.

\section{Proof of Proposition \ref{lattice}}
Let $\alpha$ be a Salem number of degree $d$ and consider $\sigma_1, \dots ,\sigma_d : \Q(\alpha) \hookrightarrow \C$ its complex embeddings (we change the notation from the last section). Let $Q$ be the minimal polynomial of $\alpha$ and $L$ a positive integer. Consider the set $\mathscr{R}_L$ of $\eta / L \in \Q(\alpha)$, with $\eta \in\Z[\alpha]$, such that
\begin{itemize}
	\item  $\eta = l_0 +\dots+ l_{d-1} \alpha^{d-1} \in \Z[\alpha],\: \gcd(l_0, \dots , l_{d-1} , L) = 1$,
	\item $\forall n\geq 0$ $\text{Tr}(\eta \alpha^n) \equiv 0 \: (\text{mod } L)$.
\end{itemize}
We will prove the sets $\mathscr{R}_L$ are empty for all $L\geq 1$, except maybe for the divisors of $E(\alpha)$, where $E(\alpha)$ is the least common multiple of the denominators of the dual basis of the canonical basis $\left\{1,\alpha,\dots,\alpha^{d-1} \right\}$ of $\Q(\alpha)$. With this end, let us note that
\begin{align*}
\mathscr{P}_L &:= \left\{ \eta = l_0 +\dots+ l_{d-1}\alpha^{d-1} \in \Z[\alpha] \:|\: \gcd(l_0, \dots , l_{d-1} , L) = 1, \right.\\
&\: \left. \:\text{Tr}(\eta \alpha^n) \equiv 0  \: (\text{mod } L) \text{ for all } n\geq 0 \right\} \\
&\subseteq \{ \eta  \in \Z[\alpha] \:|\: \text{Tr}(\eta \alpha^n) \equiv 0  \: (\text{mod } L) \text{ for all } n\geq 0 \} \\
&\subseteq \mathscr{M}_L := \{ \rho \in \Q(\alpha) \:|\: \text{Tr}(\rho \alpha^n) \equiv 0  \: (\text{mod } L) \text{ for all } n\geq 0 \}.
\end{align*}

Let $\{ w_0, \dots , w_{d-1} \}$ be the dual basis (with respect to the trace form) of the basis $\{1, \alpha, \dots ,\alpha^{d-1}\}$. We claim $\mathscr{M}_L$ is a free $\Z$-submodule of $\Q(\alpha)$ with a basis $\{ Lw_0, \dots , Lw_{d-1} \}$. By definition 
\[
\text{Tr}(\alpha^i w_ j) =\delta_{ij}, \text{ for } i, j = 0, \dots , d-1 \iff w_j = \pi_1 (V^{-1} e_{j+1} ),
\]
where $\delta_{ij}$ denotes the Kronecker delta, $V$ is the Vandermonde matrix associated to the $d$ Galois conjugates of $\alpha$, $\pi_1$ denotes projection onto the first coordinate and $\{e_{1},\dots,e_{d}\}$ is the canonical basis of $\R^d$. These equations yield the $\Z$-linear independence of the $w_i$'s (by the $\Z$-linearity of the trace). Alternatively, the $w_j$'s can be calculated as (see \cite{lang2013algebraic}, Chapter 3, Proposition 2)
\[
w_j = \dfrac{m_j}{Q'(\alpha)},
\]
where
\[
\dfrac{Q(X)}{X - \alpha} = m_0 + \dots + m_{d-1}X^{d-1} ,\quad m_j \in \Z[\alpha].
\]

Let us observe from this last equation that the denominator of $Q'(\alpha)^{-1} \in \Q(\alpha)$ expressed in reduced form is a common denominator for all $w_j$'s, a fact that will be used later. Finally, to prove the $L w_j$'s generate $\mathscr{M}_L$, let $\rho \in \mathscr{M}_L$. Let $\Theta:\Q(\alpha)\hookrightarrow \C^d$ be the Minkowski embedding $\Theta (\rho) = (\sigma_1(\rho),\dots,\sigma_{d}(\rho))^T$. By definition of $\mathscr{M}_L$, for some $p_1,\dots,p_d \in \Z$ we have
\begin{align*}
V \Theta(\rho) &= ( L p_1, \dots , L p_d )^T,  \\
\iff\Theta(\rho) = L V^{-1} ( p_1, \dots , p_d )^T &= p_1 L \Theta(w_0) + \dots + p_d L \Theta (w_{d-1}) \\
\iff\rho &= p_1 L w_0 + \dots + p_d L w_{d-1}.
\end{align*}

This completes the proof of the claim. Each $w_j$ can be expressed as $\eta_j / E$ (possibly not in reduced form) with $\eta_j \in \Z[\alpha]$ and $E=E(\alpha)$ only depending on $\alpha$ (observation of the last paragraph). Suppose $L\nmid E$, and let $l>1$ be some factor of $L$ which is not a factor of $E$. Then, for every $\Z$-linear combination of the $Lw_j$'s, the coefficients of the numerator of this linear combination (written in the canonical basis of the number field) will have $l$ as a common factor, i.e., none of these $\Z$-linear combinations will belong to $\mathscr{P}_L$ (because of the coprimality condition). In consequence, this set is empty and $\mathscr{R}_L$ is empty too. This yields the conclusion of Proposition \ref{lattice}.

\section{Dependence of $r_0$}
The aim of this section is to prove Proposition \ref{propr_0}, which provides a lower bound for $r_0$ appearing in Theorems \ref{theorem1} and \ref{theorem2}. To accomplish this we will find an upper bound for the error in the approximation of the integrals of Corollary \ref{cor}. We are only able to do this in the case $\deg(\alpha)=4$. First, let us recall the notion of type of a real number: for a real number $\theta$ we say its \textit{type} is at most $\tau \geq 1$ if there exists a positive constant $c(\theta)$ such that for every $p\in \Z$ and $q\in\N$ we have
\[
\abs{q\theta - p} \geq \dfrac{c(\theta)}{q^\tau}.
\]
The set of real numbers satisfying an inequality like the one above for some $\tau\geq1$ is called the set of \textit{Diophantine numbers}.\\

We use a general bound for linear forms in logarithms to show that $\theta_{1}$ is Diophantine. We summarize this calculation in the next
\begin{lemma}
	Let $\alpha_1=e^{2\pi i \theta_{1}}$ be an algebraic number on the unit circle which is not a root of unity. Then $\theta_{1}$ is Diophantine and its type is bounded from above by some explicit constant $\tau_{\alpha_1}$.
\end{lemma}
\begin{proof}
	Let us denote by $p_q$ the nearest integer to $q\theta_{1}$, i.e., $\norm{q\theta_{1}}_{\R/\Z} = \abs{q\theta_{1} - p_q}$. Note that
	\begin{align*}
	\abs{\alpha_1^q -1} &= \abs{e^{2\pi i q\theta_{1}} - e^{2\pi i p_q}}\\
	&\leq 2\pi \abs{q\theta_{1} - p_q}. 
	\end{align*}
	Now we use Theorem 2.2 from \cite{bugeaud2018linear} to deduce that for every $q\in \Z\setminus\{0\}$ holds $\log \abs{\alpha_1^q -1} \geq - \tau_{\alpha_1} \log(eq)$, for some explicit constant $\tau_{\alpha_1}>1$ depending only on $\alpha_1$, which finally implies 
	\[
	\abs{q\theta_{1} - p_q} \geq \dfrac{e^{-\tau_{\alpha_1}}}{2\pi q^{\tau_{\alpha_1}}} =: \dfrac{c_{\alpha_1}}{ q^{\tau_{\alpha_1}} }.
	\] 
	In other words, the type of $\theta_1$ is $\leq \tau_{\alpha_1}$.
\end{proof}
As we have already mentioned, our aim is to understand how fast is the approximation of the integrals appearing in Corollary \ref{cor} by Birkhoff sums. The classic way to do this is to use the Koksma-Hlawka inequality for integrals of dimension greater than one. Unfortunately, the integrand of the integrals in Corollary \ref{cor} are given by the functions
\[
F(x_1,\dots,x_m)=\mathds{1}_{_J} \left(\dfrac{-2\mathcal{R}(x_1,\dots,x_m)+a_p}{L}\right),
\]
$p=0,\dots,P-1$, which we do not expect to have bounded \textit{variation in the sense of Hardy and Krause} (see \cite{kuipers2012uniform}, Chapter 2, Definition 5.2). So we restrict ourselves to dimension 1 ($m=1$, so $\deg(\alpha)=4$), where we can prove that the variation of this function is bounded explicitly. This reduces ourselves to a classic result of Koksma which we recall next. We start with the definition of \textit{discrepancy}.
\begin{definition}
	Let $N\geq1$ and $(u_n)_{n\geq 1}\subset[0,1]$. The \textit{discrepancy} of the sequence $u_n$ is defined by
	\[
	D_N(u_n) = \sup_{0\leq \mathfrak{a}<\mathfrak{b}\leq1} \left| \dfrac{1}{N}\sum_{n=1}^N \mathds{1}_{_{[\mathfrak{a},\mathfrak{b}]}} (u_n) - (\mathfrak{b}-\mathfrak{a}) \right|.
	\]
\end{definition}	
By allowing only sets of the form $[0,\mathfrak{b}]$ in the definition of $D_N$, we obtain the \textit{star-discrepancy} $D^*_N$. It is obvious from the definition that $D^*_N(u_n) \leq D_N(u_n)$.

It is known (see \cite{kuipers2012uniform}, Chapter 2, Theorem 3.2) that the discrepancy of the sequence $(n\widetilde{\theta_{1}})_{n\geq1}$, where $\widetilde{\theta_{1}}=P\theta_{1}$, is bounded as
\[
D_N(n\widetilde{\theta_{1}}) \leq \mathfrak{D}N^{-1/\tau},
\]
where $\tau$ is any strict upper bound for the type of $\widetilde{\theta_{1}}$ and $\mathfrak{D}$ is an absolute constant. This fact is derived from the Erd\H{o}s-Turan inequality (see \cite{kuipers2012uniform}). The inequality above holds if we replace the upper bound for the type of $\widetilde{\theta_{1}}$ by an upper bound for the type of $\theta_{1}$, since $\widetilde{\theta_{1}}=P\theta_{1}$ (see the definition of upper bound for the type). Now we recall the Koksma inequality.
\begin{theorem}[see \cite{kuipers2012uniform}, Chapter 2, Theorem 5.1]
	Let $F:[0,1] \longrightarrow \R$ be a function of bounded variation $V(F) < +\infty$. For any sequence $(u_n)_{n\geq1} \subseteq[0,1]$ and $N\geq 1$, holds
	\[
	\left|\dfrac{1}{N}\sum_{n=1}^{N} F(u_n) - \int F(x)dx\right| \leq V(F)D^*_N(u_n).
	\]
\end{theorem}
Now we use the Koksma inequality for every function 
\[
F_p(x) = \mathds{1}_{_{J(2\delta)}}\left(\dfrac{-2\mathcal{R}(x)+a_p}{L}\right) \quad (p=0,\dots,P-1), 
\]
where $\delta=\delta(A,B,C)$ is the one obtained from Lemmas \ref{lemma1} to \ref{lemma4} for fixed $A,B,C >1$. Passing from $\delta$ to $2\delta$ will be useful in the calculations below to manage the term corresponding to $\abs{\sigma_0(\eta)}$. Note that by replacing $\delta$ with $2\delta$ in these lemmas, we can conclude that the integral of each $F_p$ is greater or equal to $2/5$: it is enough to look at the inequalities at the end of the proofs.

Indeed, from the proof of Lemma \ref{lemma1}, if $\delta>0$ satisfies
\[
1/2 < 1 - 2\delta - 0.4 - 4\delta c^{-1}\abs{\tilde{\eta}\beta^{n_2}}^{d},
\]
then $\delta' = 2\delta$ satisfies
\[
1 < 2 - 2\delta' - 0.8 - 4\delta' c^{-1}\abs{\tilde{\eta}\beta^{n_2}}^{d},
\]
i.e.,
\[
2/5 < 1 - 2\delta' - 0.4 - 4\delta' c^{-1}\abs{\tilde{\eta}\beta^{n_2}}^{d}.
\]

Same argument is valid in the case of Lemmas \ref{lemma2} and \ref{lemma4}. In the case of Lemma \ref{lemma3} the integral is still equal to 1 after this change.

The variation of $F_p$ is less or equal to $8\mathfrak{H}$, where $\mathfrak{H} =\ceil{2H} +2$ and $H = \max_{x\in\R/\Z} \left|\mathcal{R}(x)\right|$. Indeed, on any of the intervals $[0,1/4], [1/4,1/2], [1/2,3/4], [3/4,1]$ the number of intervals where the indicator function $\mathds{1}_{_{J(2\delta)}}$ is equal to one is bounded by $\ceil{2H/L}+2$ and consequently the variation of $F_p$ on any of these intervals is bounded by $2(\ceil{2H/L}+2)\leq 2\ceil{2H}+4 = 2\mathfrak{H}$. Since this is valid for all intervals, the total variation is bounded by $8\mathfrak{H}$. See Figure \ref{fig1}.

\begin{figure}[h]
	\centering
	\includegraphics[width=0.7\linewidth]{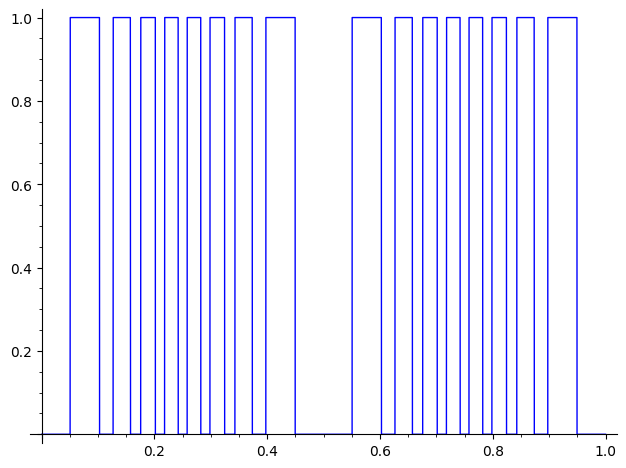}
	\caption{Graph of the function $F(x) = \mathds{1}_{[\delta,1-\delta]}(-4\cos(2\pi x))$ for $\delta = 0.2$. }
	\label{fig1}
\end{figure}

Define $n_0:= \ceil{\log_{\alpha}(\abs{\sigma_0(\eta)}/\delta)}$ and $N_p= \#\left\{0\leq n \leq N \:|\:  n\equiv p\: (\text{mod }P)\right\} \geq \floor{N/P}$. This fact and the Koksma inequality  allows us to affirm that for any $N>n_0$ and $p=0,\dots,P-1$
\begin{gather*}
\int_{\R/\Z} \mathds{1}_{_{J(2\delta)}}\left(\dfrac{-2\mathcal{R}(x)+a_p}{L}\right)dx - \dfrac{8\mathfrak{D}\mathfrak{H}}{\floor{N/P}^{1/\tau}} - \dfrac{n_0}{\floor{N/P}}\\
\leq\\
\dfrac{\#\left\{n_0\leq n\leq N,\: n\equiv p\: (\text{mod }P) \middle| \dfrac{-2\mathcal{R}_n+a_p}{L} \:(\text{mod }1) \in J(2\delta) \right\} }{\floor{N/P}}\\
\leq\\
\dfrac{\#\left\{n_0\leq n\leq N,\: n\equiv p\: (\text{mod }P)\middle| \dfrac{-2\mathcal{R}_n+a_p}{L} -\sigma_0(\eta)\alpha^{-n} \:(\text{mod }1) \in J(\delta) \right\} }{\floor{N/P}}\\
\leq\\
\dfrac{\#\left\{n_0\leq n\leq N,\: n\equiv p\: (\text{mod }P)\middle| \{\eta\alpha^n\} \in J(\delta) \right\} }{\floor{N/P}}\\
\leq\\
\dfrac{\#\left\{0\leq n\leq N,\: n\equiv p\: (\text{mod }P)\middle| \{\eta\alpha^n\} \in J(\delta) \right\} }{\floor{N/P}}.
\end{gather*}

Let us add these inequalities for $p=0,\dots,P-1$ and multiply by $\floor{N/P}/N$ both sides, yielding
\begin{gather*}
\dfrac{1}{P}\sum_{p=0}^{P-1}\int_{\R/\Z} \mathds{1}_{_{J(2\delta)}}\left(\dfrac{-2\mathcal{R}(x)+a_p}{L}\right)dx - \dfrac{1}{N/P}- \dfrac{8\mathfrak{D}\mathfrak{H}}{(N/P)^{1/\tau}} - \dfrac{n_0}{N/P}\\
\leq\\
\dfrac{\#\left\{0\leq n\leq N \middle| \{\eta\alpha^n\} \in J(\delta) \right\} }{N}.
\end{gather*}
Since each integral term on the left-hand side is greater than $2/5$ we conclude the right-hand side is greater than $1/3$ as soon as $N\geq N_0$, with $N_0$ defined by any integer solution to
\[
 \dfrac{8\mathfrak{D}\mathfrak{H}}{{(N_0/P)}^{1/\tau}} + \dfrac{n_0}{N_0/P} + \dfrac{P}{N_0} < \dfrac{1}{15}.
\]

In particular, we may choose any $N_0 \geq P\max\left(45n_0, (360\mathfrak{D}\mathfrak{H})^{\tau},45\right)$. In this manner, for $\eta =\omega\kappa = \dfrac{1}{L}(l_0 + \dots + l_{d-1}\alpha^{d-1}) \in \Q(\alpha)$, we conclude that we can take $r_0=r_0(\omega)$ in Theorem \ref{theorem1} as
\[
r_0 = \dfrac{c_\alpha}{\alpha^{P\max(\ceil{\log_{\alpha}(\abs{\sigma_0(\eta)}/\delta)},\mathfrak{H}^{\tau})}},
\]
for certain explicit constant $c_\alpha>0$, deduced from the relation $R_0 = \alpha^{N_0+1}$ and Lemmas \ref{Spectralbound}, \ref{lemmaprod}. Since we have $P\leq L^4$, we have proved the next

\begin{prop}\label{propr_0}
	Let $A,B,C > 1$ and $\omega \in \Q(\alpha)\setminus{\{0\}}$ satisfying the conditions of Theorem \ref{theorem1}. Suppose $\alpha$ is a Salem number of degree equal 4 and let $\alpha_1=e^{2\pi i \theta_1}$ be the Galois conjugate on the upper half of the unit circle. Then $\theta_{1}$ is Diophantine and if $\tau$ is an upper bound for its type, there exists a constant $c_\alpha>0$ only depending on $\alpha$ such that $r_0(\omega)$ appearing in Theorem \ref{theorem1} satisfies
	\[
	r_0(\omega) > \dfrac{c_\alpha}{\alpha^{A^4\max(\ceil{\log_{\alpha}(C/\delta)},\mathfrak{H}^\tau )} },
	\]
	where $\mathfrak{H} =\ceil{2H} +2$, $H = \sqrt{U(\theta_{1})^2 + V(\theta_{1})^2}$ (see Corollary \ref{cor} for definition of $U$ and $V$) and $\delta=\delta(A,B,C) > 0$ comes from Lemmas \ref{lemma1} to \ref{lemma4}.
\end{prop}
Almost identical calculations lead to an expression bounding $r_0(\omega)$ in the case of Theorem \ref{theorem2}, so we omit them.

\begin{section}{Appendix}
	In this section we give the details of the proof of Lemmas \ref{lemma2} and \ref{lemma4}.
	
	\textit{Proof of Lemma \ref{lemma2}}: let $\mathfrak{a}=\delta$, $\mathfrak{b}=1-\delta$ for $\delta>0$ a parameter we will choose later suitably to satisfy the conclusion of the lemma. Let $z(\vec{x})=-2\mathcal{R}(\vec{x})/L$. Then
	\begin{align*}
	\int \mathcal{B}_N(-z+\mathfrak{a})d\vec{x} &= \int \mathcal{V}_N(-z+\mathfrak{a}) + \dfrac{1}{2(N+1)}\Delta_{N+1}(-z+\mathfrak{a})d\vec{x} \\
	&= \int \dfrac{1}{N+1} \sum_{k=1}^N f\left(\dfrac{k}{N+1}\right)\sin(2\pi k(-z+\mathfrak{a}))d\vec{x}\\
	& + \: \int \dfrac{1}{2(N+1)} \left\{ 1 + \sum_{0<\abs{k}\leq N+1} \left(1 - \dfrac{\abs{k}}{N+1}\right)e^{2\pi i k (-z+\mathfrak{a})} \right\}d\vec{x} \\
	&= \underbrace{ \dfrac{-1}{N+1} \sum_{k=1}^N f\left(\dfrac{k}{N+1}\right)\sin(2\pi k\mathfrak{a}) \prod_{j=1}^m J_0\left(\dfrac{4\pi kH_{j}}{L}\right) }_{= \,(1)} \\
	& + \: \underbrace{ \dfrac{1}{2(N+1)} \left\{ 1 + \sum_{0<\abs{k}\leq N+1} \left(1 - \dfrac{\abs{k}}{N+1}\right)e^{2\pi i k\mathfrak{a}} \prod_{j=1}^m J_0\left(\dfrac{4\pi kH_{j}}{L}\right) \right\} }_{= \,(2)}.
	\end{align*}
	
	From the hypothesis follows that there exists $j^*\in\left\{1,\dots,m\right\}$ such that $2H_{j^*}/L\geq\delta_1(\beta)/2m$. Set $\mathfrak{g} = H_{j^*}/L$ and note $\mathfrak{g}^{-1} \leq 4m\delta_1(\beta)^{-1} = \mathfrak{h}^{-1}$. Define for each nonnegative integer $l$ the number $k_l = \floor{\mathfrak{h}^{-1}l}$. From now on, we will consider $N$ of the form $\floor{\mathfrak{h}^{-1}T}$, for each $T\in\N$. We use below the inequality for the Bessel function from Proposition \ref{Bessel}.
	\begin{align*}
	\abs{(2)} &\leq \dfrac{1}{2(N+1)} \left\{ 1 + 2\sum_{k=1}^{N+1} \left(1 - \dfrac{k}{N+1}\right) \prod_{j=1}^m \left|J_0\left(\dfrac{4\pi kH_j}{L}\right)\right| \right\} \\
	&\leq \dfrac{1}{2(N+1)} \left\{ 1 + 2\sum_{l=0}^{T-1}\sum_{k=k_l+1}^{k_{l+1}} \left(1 - \dfrac{k}{N+1}\right) \prod_{j=1}^m\left|J_0\left(\dfrac{4\pi kH_j}{L}\right)\right| \right\} \\
	&\leq \dfrac{1}{2(N+1)} \left\{ 1 + 2\sum_{k=1}^{k_{1}} \left(1 - \dfrac{k}{N+1}\right) \prod_{j=1}^m \underbrace{\left|J_0\left(\dfrac{4\pi kH_j}{L}\right)\right|}_{\leq 1} \right.\\
	& \left.+\: 2\sum_{l=1}^{T-1}\sum_{k=k_l+1}^{k_{l+1}} \left(1 - \dfrac{k}{N+1}\right) \underbrace{\prod_{j=1}^m \left|J_0\left(\dfrac{4\pi kH_j}{L}\right)\right|}_{\leq 1/\sqrt{2}\pi(k\mathfrak{g})^{1/2}} \right\}\\
	&\leq \dfrac{1}{2(N+1)} \left\{ 1 +  2\mathfrak{h}^{-1} + \dfrac{\sqrt{2}}{\pi}\sum_{l=1}^{T-1}\sum_{k=k_l+1}^{k_{l+1}} \underbrace{(k\mathfrak{g})^{-1/2}}_{\leq \:l^{-1/2}} \right\} \\
	&\leq \dfrac{1}{2(N+1)} \left\{ 1 + 2\mathfrak{h}^{-1} + \dfrac{\sqrt{2}}{\pi}\sum_{l=1}^{T-1} l^{-1/2} \underbrace{(k_{l+1}-k_l)}_{\leq \: \mathfrak{h}^{-1}} \right\} \\
	&\leq \dfrac{1}{2(N+1)} \left\{ 1 + 2\mathfrak{h}^{-1} + \dfrac{\sqrt{2}\mathfrak{h}^{-1}}{\pi}O(\sqrt{T}) \right\} \\
	&= O(1/\sqrt{T})
	\end{align*}
	Similarly,
	\[
	\int \mathcal{B}_N(z-\mathfrak{b})d\vec{x} = \dfrac{1}{N+1} \sum_{k=1}^N f\left(\dfrac{k}{N+1}\right)\sin(2\pi k\mathfrak{b}) \prod_{j=1}^m J_0\left(\dfrac{4\pi kH_j}{L}\right) + O(1/\sqrt{T}).
	\]
	
	Having estimated the integral of the Beurling polynomials, we continue with the Selberg ones:
	\begin{gather*}
	\left|\int \mathcal{B}_N(z-\mathfrak{b}) + \mathcal{B}_N(-z+\mathfrak{a}) d\vec{x}\,\right| \\
	\leq \\
	\left| \dfrac{1}{N+1}  \sum_{k=1}^N f\left(\dfrac{k}{N+1}\right) (\sin(2\pi k\mathfrak{b}) - \sin(2\pi k\mathfrak{a})) \prod_{j=1}^{m} J_0\left(\dfrac{4\pi kH_j}{L}\right) \right| + O(1/\sqrt{T}) \\
	= \\
	\underbrace{ \left| \dfrac{2}{N+1}  \sum_{k=1}^N f\left(\dfrac{k}{N+1}\right) \sin(2\pi k\delta) \prod_{j=1}^{m} J_0\left(\dfrac{4\pi kH_j}{L}\right) \right| }_{=(3)} + \:O(1/\sqrt{T}).
	\end{gather*}
	
	Fix $\epsilon>0$ and let us take $\xi<1/2$ such that $\pi\xi/\sin(\pi\xi) \leq 1+\epsilon$ (recall the definition and property of $f$ in Definition \ref{defValeer}). Then, for $T$ big enough, we obtain
	\begin{align*}
	(3) &\leq \dfrac{2}{N+1} \sum_{k=1}^N \left|f\left(\dfrac{k}{N+1}\right)\sin(2\pi k\delta)\prod_{j=1}^m J_0\left(\dfrac{4\pi kH_j}{L}\right) \right| \\
	&\leq \dfrac{2}{N+1}\left[\: \sum_{k=1}^{k_{1}} \abs{ f\left(\dfrac{k}{N+1}\right) } \underbrace{\abs{ \sin(2\pi k\delta) }}_{\leq2\pi k\delta} \prod_{j=1}^m \underbrace{\left|J_0\left(\dfrac{4\pi kH_j}{L}\right)\right|}_{\leq 1} \right.\\
	&\quad + \sum_{l=1}^{\floor{\xi T}}\sum_{k=k_l+1}^{k_{l+1}} \underbrace{\abs{f\left(\dfrac{k}{N+1}\right)}}_{\leq \left|f(l\mathfrak{h}^{-1}/(N+1))\right|} \underbrace{\abs{\sin(2\pi k\delta)}}_{\leq 1} \underbrace{ \prod_{j=1}^m\left|J_0\left(\dfrac{4\pi kH_j}{L}\right)\right|}_{\leq 1/\sqrt{2}\pi(k\mathfrak{g})^{1/2}} \\
	&\quad \left. + \sum_{l=\floor{\xi T}+1}^{T-1}\:\sum_{k=k_l+1}^{k_{l+1}} \abs{f\left(\dfrac{k}{N+1}\right)} \underbrace{\abs{\sin(2\pi k\delta)}}_{\leq 1} \underbrace{\prod_{j=1}^m \left| J_0\left(\dfrac{4\pi kH_j}{L}\right)\right|}_{\leq 1/\sqrt{2}\pi(k\mathfrak{g})^{1/2}} \: \right]\\	
	&\leq  \dfrac{4\pi\delta}{N+1} \sum_{k=1}^{k_{1}} \left(\dfrac{\pi\xi}{\sin(\pi\xi)}\dfrac{N+1}{\pi k}+\dfrac{1}{\pi}\right)k \\
	&\quad + \dfrac{\sqrt{2}}{\pi(N+1)}  \sum_{l=1}^{\floor{\xi T}} \abs{f\left(\dfrac{l\mathfrak{h}^{-1}}{N+1}\right)} \sum_{k=k_l+1}^{k_{l+1}} \underbrace{(k\mathfrak{g})^{-1/2}}_{\leq \:l^{-1/2}} \\
	&\quad + \dfrac{\sqrt{2}}{\pi(N+1)} \sum_{l=\floor{\xi T}+1}^{T-1} \left( \dfrac{1-\xi}{\sin(\pi(1-\xi))} + \dfrac{1}{\pi}\right) \:\sum_{k=k_l+1}^{k_{l+1}} \underbrace{(k\mathfrak{g})^{-1/2}}_{\leq \:l^{-1/2}}\\
	&\leq 4(1+\epsilon)\delta \mathfrak{h}^{-1} + O(1/T) + \dfrac{\sqrt{2}\mathfrak{h}^{-1}}{\pi(N+1)}\sum_{l=1}^{\floor{\xi T}} \left(\dfrac{\pi\xi}{\sin(\pi\xi)}\dfrac{(N+1)\mathfrak{h}}{\pi l}+\dfrac{1}{\pi}\right)l^{-1/2} \\
	&\quad + O(1/\sqrt{T}) \\
	&\leq 4(1+\epsilon)\delta \mathfrak{h}^{-1} + \dfrac{\sqrt{2}(1+\epsilon)\zeta\left(\dfrac{3}{2}\right)}{\pi^2}  + O(1/\sqrt{T}),
	\end{align*}
	where $\zeta$ denotes the Riemann zeta function. In the same manner,
	\[
	\left|\int \mathcal{B}_N(-z+\mathfrak{b}) + \mathcal{B}_N(z-\mathfrak{a}) d\vec{x}\,\right| \leq 4(1+\epsilon)\delta \mathfrak{h}^{-1} + \dfrac{\sqrt{2}(1+\epsilon)\zeta\left(\dfrac{3}{2}\right)}{\pi^2} + O(1/\sqrt{T}).
	\]
	
	It may be checked that $\dfrac{\sqrt{2}\zeta(3/2)}{\pi^2} < 0.4$. Since $\epsilon>0$ was arbitrarily chosen, by taking $T\to\infty$ we conclude by means of \ref{propertySelberg} that
	\[
	\left| \int_{(\R/\Z)^m} \mathds{1}_{_{J(\delta)}}(z(\vec{x}))d\vec{x} - \abs{J(\delta)}\right| \leq  4\delta \mathfrak{h}^{-1} + 0.4,
	\]
	which implies 
	\[
	1 - 2\delta - 0.4 - 4\delta \mathfrak{h}^{-1} \leq \int_{(\R/\Z)^m} \mathds{1}_{_{J(\delta)}}(z(\vec{x}))d\vec{x}.
	\]
	Then, by solving  $1/2 < 1 - 2\delta - 0.4 - 16\delta m\delta_1(\beta)^{-1}$, we obtain the desired $\delta = \delta(L)$.
	
	\textit{Proof of Lemma \ref{lemma4}}: let $\mathfrak{a}=\delta$, $\mathfrak{b}=1-\delta$ for $\delta>0$ a parameter we will choose later suitably to satisfy the conclusion of the lemma. Let $z(\vec{x}) = \dfrac{-2\mathcal{R}(\vec{x}) + l}{L}$. As before, we have
	\begin{align*}
	\int \mathcal{B}_N(-z+\mathfrak{a})d\vec{x} &= \underbrace{ \dfrac{-1}{N+1} \sum_{k=1}^N f\left(\dfrac{k}{N+1}\right)\sin(2\pi k\mathfrak{a}) \prod_{j=1}^m J_0\left(\dfrac{4\pi kH_{j}}{L}\right) }_{= \,(1)} \\
	& + \: \underbrace{ \dfrac{1}{2(N+1)} \left\{ 1 + \sum_{0<\abs{k}\leq N+1} \left(1 - \dfrac{\abs{k}}{N+1}\right)e^{2\pi i k\mathfrak{a}} \prod_{j=1}^m J_0\left(\dfrac{4\pi kH_{j}}{L}\right) \right\} }_{= \,(2)}.
	\end{align*}
	
	From the hypothesis follows that there exists $j^*\in\left\{1,\dots,m\right\}$ such that $H_{j^*}\geq1/4m$. Set $\mathfrak{g} = H_{j^*}/L$ and note $\mathfrak{g}^{-1} \leq 4mL = \mathfrak{h}^{-1}$. Define for each nonnegative integer $l$ the number $k_l = \floor{\mathfrak{h}^{-1}l}$. From now on, we will consider $N$ of the form $\floor{\mathfrak{h}^{-1}T}$, for each $T\in\N$. The same calculations of the proof above lead to $(2)=O(1/\sqrt{T})$. We conclude in the same manner that
	\begin{gather*}
	\left|\int \mathcal{B}_N(z-\mathfrak{b}) + \mathcal{B}_N(-z+\mathfrak{a}) d\vec{x}\,\right| \\
	\leq \\
	\left| \dfrac{1}{N+1}  \sum_{k=1}^N f\left(\dfrac{k}{N+1}\right) (\sin(2\pi k\mathfrak{b}) - \sin(2\pi k\mathfrak{a})) \prod_{j=1}^{m} J_0\left(\dfrac{4\pi kH_j}{L}\right) \right| + O(1/\sqrt{T}) \\
	= \\
	\underbrace{ \left| \dfrac{2}{N+1}  \sum_{k=1}^N f\left(\dfrac{k}{N+1}\right) \sin(2\pi k\delta) \prod_{j=1}^{m} J_0\left(\dfrac{4\pi kH_j}{L}\right) \right| }_{=(3)} + \:O(1/\sqrt{T}).
	\end{gather*}
	
	We repeat the procedure to bound (3) used in the proof above, yielding
	\[
	\left| \int_{(\R/\Z)^m} \mathds{1}_{_{J(\delta)}}(z(\vec{x}))d\vec{x} - \abs{J(\delta)}\right| \leq  4\delta \mathfrak{h}^{-1} + 0.4.
	\]
	Then, by solving  $1/2 < 1 - 2\delta - 0.4 - 16\delta mL$, we obtain the desired $\delta = \delta(L)$.
\end{section}
\bibliographystyle{amsalpha}
\bibliography{bibliografia.bib}
%\begin{thebibliography}{10}
%\end{thebibliography}
\end{document}